\documentclass{amsart}
 
\usepackage{amsmath}
\usepackage{amssymb}
\usepackage{amsfonts}

\usepackage{alltt}
\usepackage{amsthm}
\usepackage{bm}
\usepackage{mathrsfs}
\usepackage{graphicx}
\usepackage{color}
	\usepackage{tikz,pgfplots}
\usepackage{mathtools}
\usepackage[T1]{fontenc}
\usepackage[utf8]{inputenc}
\usepackage{xcolor}
\usepackage[colorlinks,linkcolor=black,urlcolor=black,citecolor=black,anchorcolor=black]{hyperref}

\newcommand{\De}{\mathscr{D}}

\newcommand{\C}{\mathbb C}
\newcommand{\D}{\mathbb D}
\newcommand{\Hu}{\mathbb H}
\newcommand{\R}{\mathbb R}
\newcommand{\T}{\mathbb T}
\newcommand{\N}{\mathbb N}
\newcommand{\Z}{\mathbb Z}

\newcommand{\supp}{\operatorname{supp}}
\newcommand{\re}{\operatorname{Re}}
\newcommand{\im}{\operatorname{Im}}

\newcommand{\spec}{\operatorname{Spec}}

\numberwithin{equation}{section}

\newtheorem{theorem}{Theorem}[section]
\newtheorem{prop}[theorem]{Proposition}
\newtheorem{lemma}[theorem]{Lemma}
\newtheorem{cor}[theorem]{Corollary}

\newtheorem*{thm*}{Theorem}
\newtheorem*{lemma*}{Lemma}
\newtheorem*{prop*}{Proposition}

\theoremstyle{definition}
\newtheorem{dfn}[theorem]{Definition}

\newtheorem*{dfn*}{Definition}

\theoremstyle{remark}
\newtheorem{remark}[theorem]{Remark}

\makeatletter
\let\save@mathaccent\mathaccent
\newcommand*\if@single[3]{%
  \setbox0\hbox{${\mathaccent"0362{#1}}^H$}%
  \setbox2\hbox{${\mathaccent"0362{\kern0pt#1}}^H$}%
  \ifdim\ht0=\ht2 #3\else #2\fi
  }
\newcommand*\rel@kern[1]{\kern#1\dimexpr\macc@kerna}
\newcommand*\widebar[1]{\@ifnextchar^{{\wide@bar{#1}{0}}}{\wide@bar{#1}{1}}}
\newcommand*\wide@bar[2]{\if@single{#1}{\wide@bar@{#1}{#2}{1}}{\wide@bar@{#1}{#2}{2}}}
\newcommand*\wide@bar@[3]{%
  \begingroup
  \def\mathaccent##1##2{%
    \let\mathaccent\save@mathaccent
    \if#32 \let\macc@nucleus\first@char \fi
    \setbox\z@\hbox{$\macc@style{\macc@nucleus}_{}$}%
    \setbox\tw@\hbox{$\macc@style{\macc@nucleus}{}_{}$}%
    \dimen@\wd\tw@
    \advance\dimen@-\wd\z@
    \divide\dimen@ 3
    \@tempdima\wd\tw@
    \advance\@tempdima-\scriptspace
    \divide\@tempdima 10
    \advance\dimen@-\@tempdima
    \ifdim\dimen@>\z@ \dimen@0pt\fi
    \rel@kern{0.6}\kern-\dimen@
    \if#31
      \overline{\rel@kern{-0.6}\kern\dimen@\macc@nucleus\rel@kern{0.4}\kern\dimen@}%
      \advance\dimen@0.4\dimexpr\macc@kerna
      \let\final@kern#2%
      \ifdim\dimen@<\z@ \let\final@kern1\fi
      \if\final@kern1 \kern-\dimen@\fi
    \else
      \overline{\rel@kern{-0.6}\kern\dimen@#1}%
    \fi
  }%
  \macc@depth\@ne
  \let\math@bgroup\@empty \let\math@egroup\macc@set@skewchar
  \mathsurround\z@ \frozen@everymath{\mathgroup\macc@group\relax}%
  \macc@set@skewchar\relax
  \let\mathaccentV\macc@nested@a
  \if#31
    \macc@nested@a\relax111{#1}%
  \else
    \def\gobble@till@marker##1\endmarker{}%
    \futurelet\first@char\gobble@till@marker#1\endmarker
    \ifcat\noexpand\first@char A\else
      \def\first@char{}%
    \fi
    \macc@nested@a\relax111{\first@char}%
  \fi
  \endgroup
}
\makeatother

\makeatletter
\def\@setauthors{%
  \begingroup
  \def\thanks{\protect\thanks@warning}%
  \trivlist
  \centering\large \@topsep30\p@\relax
  \advance\@topsep by -\baselineskip
  \item\relax
  \author@andify\authors
  \def\\{\protect\linebreak}%
  \authors%
  \ifx\@empty\contribs
  \else
    ,\penalty-3 \space \@setcontribs
    \@closetoccontribs
  \fi
  \endtrivlist
  \endgroup
}
\def\@settitle{\begin{center}%
  \baselineskip14\p@\relax
    \normalfont\LARGE
  \@title
  \end{center}%
}
\makeatother

\pgfplotsset{compat=newest}

\begin{document}

\title[Uniqueness theorems]{Uniqueness theorems for 
weighted harmonic functions\\ in the upper half-plane}

\date{\today}

\author{Anders Olofsson}
\address[Anders Olofsson]{Centre for Mathematical Sciences, Lund University, Box 118, SE-221 00 Lund, Sweden}
\email{olofsson@maths.lth.se}

\author{Jens Wittsten}
\address[Jens Wittsten]{Centre for Mathematical Sciences, Lund University, Box 118, SE-221 00 Lund, Sweden, and Department of Engineering, University of Bor{\aa}s, SE-501 90 Bor{\aa}s, Sweden}
\email{jens.wittsten@math.lu.se}



\begin{abstract}
We consider a class of weighted harmonic functions in the open upper half-plane
known as $\alpha$-harmonic functions. 
Of particular interest is the uniqueness problem for such functions 
subject to a vanishing Dirichlet boundary value on the real line 
and an appropriate vanishing condition at infinity. 
We find that the non-classical case ($\alpha\neq0$) allows for 
a considerably more relaxed vanishing condition at infinity 
compared to the classical case ($\alpha=0$) of usual harmonic functions 
in the upper half-plane. The reason behind this dichotomy is different geometry of zero sets of certain polynomials naturally derived from the classical binomial series. 

Our findings shed new light on 
the theory of harmonic functions, for which we provide uniqueness results under vanishing conditions at infinity along a) geodesics, and b) rays emanating from the origin. The geodesic uniqueness results require vanishing on two distinct geodesics which is best possible. The ray uniqueness results involves an arithmetic condition which we analyze by introducing the concept of an admissible function of angles. We show that the arithmetic condition is to the point and that the set of admissible functions of angles is minimal with respect to a natural partial order.
\end{abstract}

\maketitle

\section{Introduction}

Let $\Hu$ be the open upper half-plane in the complex plane $\C$ 
and consider the weighted Laplace differential operator
\begin{equation}\label{HalphaLaplacian}
\Delta_{\Hu;\alpha,z}=\partial_z (\im z)^{-\alpha}\bar\partial_z,\quad z\in\Hu,
\end{equation}
where $\partial$ and $\bar\partial$ are the usual complex partial 
derivatives and $\alpha>-1$. 
Here $\im z$ is the imaginary part of the complex number $z\in\C$ and 
the function $w_{\Hu;\alpha}(z)=(\im z)^{\alpha}$ for $z\in\Hu$ 
has an interpretation of a standard weight function for $\Hu$. 
The study of differential operators of the form \eqref{HalphaLaplacian} 
is suggested by a classical paper of Paul Garabedian  \cite{Garabedian}. 

We refer to the differential operator $\Delta_{\Hu;\alpha}$ 
in \eqref{HalphaLaplacian} as 
the $\alpha$-Laplacian for $\Hu$. 
An $\alpha$-harmonic function $u$ in $\Hu$ is 
a twice continuously differentiable function $u$ in $\Hu$ 
(in symbols: $u\in C^2(\Hu)$) such that 
$$
\Delta_{\Hu;\alpha}u=0\quad  \text{in}\  \Hu.
$$ 
Notice that $\Delta_{\Hu;0}=\partial\bar\partial$ is the usual Laplacian 
and that a $0$-harmonic function in $\Hu$ is 
a harmonic function in $\Hu$ in the usual sense. 
The restriction $\alpha>-1$ on the weight parameter ensures 
good supply of $\alpha$-harmonic functions with well-behaved 
boundary values.  
 
In this paper we address the uniqueness problem for $\alpha$-harmonic 
functions in $\Hu$, that is, 
we wish to characterize the identically zero function $u\equiv 0$ 
within the class of  $\alpha$-harmonic functions in $\Hu$. 
A natural condition is that of a vanishing boundary value  
\begin{equation}\label{vanishingDirichletbdrycondition}
\lim_{\Hu\ni z\to x}u(z)=0,\quad x\in\R,
\end{equation}
on the real line $\R$. A condition of this type is often referred to 
as a vanishing Dirichlet boundary value. 
There are plenty of non-trivial $\alpha$-harmonic 
functions in $\Hu$ satisfying \eqref{vanishingDirichletbdrycondition}. 
A simple such example is the function 
$$
u(z)=(\im z)^{\alpha+1},\quad z\in\Hu,
$$  
which is $\alpha$-harmonic in $\Hu$ and satisfies  
\eqref{vanishingDirichletbdrycondition}.
In order to obtain satisfactory results on the uniqueness problem above 
it is thus natural to complement \eqref{vanishingDirichletbdrycondition}   
with some condition(s) taking into account 
the behavior of the $\alpha$-harmonic function $u$ at infinity. 
Here we think of the point at infinity $\infty$ as a boundary point of 
$\Hu$ in the extended complex plane $\C_\infty=\C\cup\{\infty\}$. 
In the case of usual harmonic functions in $\Hu$ ($\alpha=0$) 
this problem setup is classical. 
We mention here a recent contribution by
Carlsson and Wittsten \cite{carlsson2016dirichlet} 
concerned with a uniqueness result in this flavor 
for $\alpha$-harmonic 
functions in $\Hu$ with $\alpha>-1$.
This result of Carlsson and Wittsten has served as a guidance 
for the present investigations.

We say that a function $u$ in $\Hu$ is of temperate growth at infinity 
if it satisfies an estimate of the form
$$
\lvert u(z)\rvert \leq C(\lvert z\rvert^2/\im(z))^N
$$   
for $z\in\Hu$ with $\lvert z\rvert>R$, 
where $C$, $R$ and $N$ are positive constants. 
We refer to the parameter $N$ as an order of growth at infinity 
for the function $u$.

A first main result concerns $\alpha$-harmonic functions $u$ in $\Hu$ 
that are of temperate growth at infinity. 
We prove that such a function $u$ satisfies  
\eqref{vanishingDirichletbdrycondition} if and only if  
it has the form 
\begin{equation}\label{alphaobstructionfunction}
u(z)=\sum_{k=0}^n  c_k(\im z)^{\alpha+1} p_{k,\alpha}(z),\quad z\in\Hu,
\end{equation}
for some $n\in\N=\{0,1,2,\dots\}$ and $c_0,\dots,c_n\in\C$, where  
\begin{equation*}
p_{k,\alpha}(z)=\sum_{j=0}^k\frac{(\alpha+1)_j}{j!}z^{k-j}\bar z^j
\end{equation*}
for $k=0,1,\dots$ (see 
Corollary \ref{Valphacharacterization}).
We point out that the polynomials $p_{k,\alpha}$ are naturally derived from  
the classical binomial series. 
We also establish \eqref{alphaobstructionfunction} 
under a weaker distributional version of 
\eqref{vanishingDirichletbdrycondition} 
(see Theorem \ref{alphaobstructionclassdistributionalversion}).

We denote by $\mathcal{V}_\alpha$ the set of all functions $u$ 
of the form \eqref{alphaobstructionfunction}   
for some $n\in\N$ and $c_0,\dots,c_n\in\C$.  
The set $\mathcal{V}_\alpha$ is naturally filtered in the sense that 
$$
\mathcal{V}_\alpha= \cup_{n=0}^\infty\mathcal{V}_{\alpha,n},
$$ 
where $\mathcal{V}_{\alpha,n}$ is the set of all functions of the form 
\eqref{alphaobstructionfunction} for some $c_0,\dots,c_n\in\C$.
The set $\mathcal{V}_{\alpha,n}$ has a natural structure of a 
complex vector space of finite dimension $n+1$. 
The space $\mathcal{V}_{\alpha,n}$ admits a natural description 
within the class $\mathcal{V}_{\alpha}$ using order of growth 
at infinity (see Lemma \ref{coefficientlemma} 
and Proposition \ref{Vanfromrelaxedgrowth}). 
The finer study of behaviors at infinity of functions 
in the class $\mathcal{V}_{\alpha}$ depends on the parameter $\alpha>-1$ 
and divides naturally into cases whether $\alpha\neq0$ or $\alpha=0$. 
The reason behind this dichotomy is a different geometry of 
zero sets for the polynomials $p_{k,\alpha}$.

We consider next the problem of characterizing the null function 
in the class $\mathcal{V}_\alpha$ using a vanishing condition at infinity. 
For parameters $\alpha>-1$ with $\alpha\neq0$, we prove that if 
$u\in\mathcal{V}_\alpha$ is such that 
\begin{equation}\label{sequencevanishingalphaneq0}
\lim_{j\to\infty} \frac{u(z_j)}{(\im(z_j))^{\alpha+1}}=0
\end{equation} 
for some sequence $\{z_j\}$ in $\Hu$ 
with $z_j \to\infty$ in  $\C_\infty$ as $j\to\infty$, then 
$u(z)=0$ for all $z\in\Hu$ (see Theorem \ref{Vauniquenessaneq0}).  
We emphasize the big freedom allowed in 
the choice of sequence $\{z_j\}$ 
in \eqref{sequencevanishingalphaneq0} above.  
This analysis leads to the following highly flexible uniqueness result 
for $\alpha$-harmonic functions in the case $\alpha\neq0$ as well as a distributional version thereof    
(see 
Theorem \ref{thm:distributionaluniquenessaneq0}).

\begin{theorem}\label{introuniquenessthm}
Let $\alpha>-1$ and $\alpha\ne0$.
Let $u$ be an $\alpha$-harmonic function in $\Hu$ which is 
of temperate growth at infinity. 
\begin{enumerate}
\item\label{classicalvanishingrealline}
Assume that \eqref{vanishingDirichletbdrycondition} holds.
\item 
Assume that there exists a sequence $\{z_j\}$ in $\Hu$ 
with $z_j \to\infty$ in $\C_\infty$ as $j\to\infty$ such that 
\eqref{sequencevanishingalphaneq0} holds.
\end{enumerate} 
Then $u(z)=0$ for all $z\in\Hu$.
\end{theorem}

We now turn our attention to the case $\alpha=0$ of usual 
harmonic functions in $\Hu$. 
A function $u$ belongs to the class $\mathcal{V}_0$ if and only if 
it has the form 
$$
u(z)=\sum_{k=1}^n c_k\im(z^k),\quad z\in\Hu,
$$
for some $n\in \Z^+=\{1,2,3,\dots\}$ and $c_1,\dots,c_n\in\C$ 
(see Proposition \ref{harmonicobstructionclass}). 
Observe that the harmonic polynomial  
$$
u_k(z)= \im(z^k),\quad z\in\C,
$$ 
vanishes on a union of $k$ lines passing through the origin. 
Notice also that the zero set of $u_k$ 
intersects the unit circle at the $2k$-th roots of unity.   
These examples make evident that the flexible uniqueness results for 
$\alpha$-harmonic functions in $\Hu$ with $\alpha\neq0$ are 
no longer true when $\alpha=0$ 
(compare with Theorem \ref{introuniquenessthm} above).  

In order to obtain satisfactory uniqueness results 
for usual harmonic functions in $\Hu$ we shall restrict 
condition \eqref{sequencevanishingalphaneq0}   
to suitable classes of curves. 
We consider two such classes of curves, namely, 
geo\-desics in $\Hu$ and rays in $\Hu$ emanating from the origin.      

By a geodesic in $\Hu$ we understand a ray in $\Hu$ which 
is parallel to the imaginary axis. We prove that if   
$u\in\mathcal{V}_0$ is such that 
\begin{equation*}
\lim_{y\to+\infty}u(x+iy)/y=0
\end{equation*}
for $x=x_j\in\R$ ($j=1,2$) with $x_1\neq x_2$, 
then $u(z)=0$ for all $z\in\Hu$ (see Theorem \ref{V0uniquenessgeodesic}).
This analysis leads to corresponding uniqueness results for usual 
harmonic functions in $\Hu$ 
(see Theorems \ref{classicaluniquenessthm2geodesics} 
and \ref{distributionaluniquenessthm2geodesics}).  
We emphasize that those geodesic uniqueness results require 
vanishing on two ($2$) distinct geodesics which is best possible. 
Together with the success results for $\alpha\neq 0$, 
this marks a significant advancement 
from an earlier uniqueness result of 
Carlsson and  Wittsten \cite[Corollary 1.9]{carlsson2016dirichlet} 
for $\alpha$-harmonic functions which required vanishing on 
an interval of geodesics.

By a ray in $\Hu$ emanating from a point $a\in\R$ we understand 
a set of the form $\{a+te^{i\theta}:\ t>0\}$, where $0<\theta<\pi$. 
We shall restrict our attention to rays emanating from the origin ($a=0$). 
In view of translation invariance of the class of harmonic functions 
this restriction is minor. 
In order to discuss vanishing of functions along rays  
we introduce a notion of admissible function of angles 
which is a function element $(E,\eta)$ with $E\subset(0,\pi)$ 
and $\eta:E\to\Z^+$ having the property that for every $k\in\Z^+$ 
there exists $\theta\in E$ such that $\sin(k\theta)\neq0$ 
and $k\geq \eta(\theta)$ 
(see Definition \ref{dfnadmissiblefcnangles}).  
We prove that if $u\in\mathcal{V}_0$ and 
there exists an admissible function of angles $(E,\eta)$ 
such that 
$$
\lim_{t\to+\infty}u(te^{i\theta})/t^{\eta(\theta)}=0
$$ 
for every $\theta\in E$, 
then $u(z)=0$ for all $z\in\Hu$ (see Theorem \ref{V0uniquenessray}).  
This analysis leads to corresponding uniqueness results for usual 
harmonic functions in $\Hu$ 
(see Theorems \ref{classicalfcnofanglesvanishing}  
and \ref{fcnofanglesvanishing}).  

The notion of admissible function of angles involves an 
arithmetic element. In order to illuminate this fact 
we mention the following result.

\begin{theorem}\label{classicalgenericrayuniquenessresult}
Let $u$ be a harmonic function in the open upper half-plane $\Hu$ 
which is of temperate growth at infinity.
\begin{enumerate}
\item 
Assume that 
\eqref{vanishingDirichletbdrycondition} holds.
\item
Assume that 
\begin{equation}\label{rayvanishingcondition}
\lim_{t\to+\infty}u(te^{i\theta})/t=0
\end{equation}
for some $0<\theta<\pi$ 
which is not a rational multiple of $\pi$.
\end{enumerate} 
Then $u(z)=0$ for all $z\in\Hu$.
\end{theorem}

The examples $u_k$ above show that the arithmetic condition 
on $\theta$ in \eqref{rayvanishingcondition} 
is to the point. 
Behind Theorem \ref{classicalgenericrayuniquenessresult} lies 
a construction of a one point admissible function of angles $(E,\eta)$, 
where $E=\{\theta\}$ and $\eta(\theta)=1$.  
The arithmetic condition on $\theta$ in \eqref{rayvanishingcondition} 
ensures that this latter function element $(E,\eta)$ is an   
admissible function of angles. 
More elaborate constructions of admissible functions of angles lead 
to corresponding uniqueness results for usual harmonic functions in $\Hu$ 
(see Corollaries \ref{genericrayvanishing}, \ref{rationalrayvanishing}, 
\ref{testthm} and \ref{testthm2}). 

In the final section we provide constructions of 
admissible functions of angles 
(see Theorems \ref{constructionfoainfinite} and \ref{constructionfoafinite}).  
The set $\mathcal{A}$ of admissible functions of angles is structured  
by a natural partial order. We show that every element in $\mathcal{A}$ 
has a lower bound which is minimal 
(see Theorem \ref{lowerboundfoa} and Lemma \ref{cfoaminimal}). 
Our constructions of admissible functions of angles yield 
precisely the minimal elements in $\mathcal{A}$ 
(see Theorem \ref{minimalelementsA}).

There is a substantial literature on 
boundary uniqueness problems for (sub-)harmonic functions 
with contributions of  
Wolf \cite{Wolf}, Shapiro \cite{ShapiroVL}, Dahlberg \cite{Dahlberg}, 
Berman and Cohn \cite{BermanCohn} and 
Borichev with collaborators \cite{BorichevCT,BorichevT} to name a few.  
As far as usual harmonic functions in $\Hu$ are concerned 
our uniqueness results supersede an earlier uniqueness result of 
Siegel and Talvila \cite[Corollary 3.1]{SiegelTalvila}. 

The $\alpha$-Laplacian (or rather its symmetric part) is also related to the Laplace-Beltrami equation in the Riemannian space defined by the metric 
$$
ds^2=x_n^{-\alpha/(n-2)}\sum_1^n dx_i^2,\quad n>2,
$$
as studied by Weinstein \cite{Weinstein} and Huber \cite{Huber}, among others. For historic reasons, solutions of said Laplace-Beltrami equation are referred to as generalized axially symmetric potentials. For a discussion on this connection and more recent applications in this direction we refer to Wittsten \cite{Wittsten}. Another related area of interest is the recent study of higher order 
Laplacians initiated by Borichev and Hedenmalm \cite{BH}.   

The present paper is rooted in previous investigations.
In an earlier paper \cite{olofsson2013poisson} we initiated a theory 
of $\alpha$-harmonic functions in the open unit disc $\D$ 
in the complex plane.  A function $u\in C^2(\D)$ is called       
$\alpha$-harmonic in $\D$ if $\Delta_{\D;\alpha}u=0$ in $\D$, 
where 
$$
\Delta_{\D;\alpha,z}=
\partial_z(1-\lvert z\rvert^2)^{-\alpha}\bar\partial _z,\quad z\in\D,
$$  
is the $\alpha$-Laplacian for $\D$ and $\alpha\in\R$. 
A main concern in this theory is the representation of an 
$\alpha$-harmonic function in $\D$ as a 
Poisson integral $u=P_\alpha[f]$ in $\D$ 
with respect to the kernel 
\begin{equation} \label{aPoissonkernel}
P_\alpha(z)=\frac{(1-\lvert z\rvert^2)^{\alpha+1}}{(1-z)(1-\bar z)^{\alpha+1}},
\quad z\in\D.
\end{equation}
We refer to the function $P_\alpha$ as 
the $\alpha$-harmonic Poisson kernel for $\D$.

We prove that a Poisson integral representation $u=P_\alpha[f]$ with $f$ 
a distribution on the unit circle $\T=\partial\D$ 
exists 
if and only if 
$u$ is $\alpha$-harmonic in $\D$, $u$ has temperate growth in $\D$ 
and a certain spectral condition is satisfied  
(see Theorem \ref{characterizationPIrepresentation}). 
This latter spectral condition is automatically satisfied 
when $\alpha$ is not a negative integer 
(see Corollary \ref{genericcharacterizationPIrepresentation}). 
For $\alpha>-1$, the distribution $f\in\De'(\T)$ has a 
natural interpretation as a (distributional) boundary limit of 
the function $u$.  
On the other hand, for $\alpha\leq-1$, 
existence of a distributional boundary limit of an $\alpha$-harmonic 
function $u$ in $\D$, forces $u$ to be analytic in $\D$ 
(see Theorem \ref{boundarybehavioralphaleq-1}). 
We thus establish Poisson integral representations 
$u=P_\alpha[f]$ with $f\in\De'(\T)$ in situations 
where a distributional boundary value of $u$ is non-existent.  
To overcome those difficulties we resort to a study of 
related hyper\-geometric functions, 
found in Section \ref{sectionhypergeometricfunction}.

A link between the 
settings of the upper half-plane $\Hu$ and the unit disc $\D$
is provided by a certain 
conformal invariance property of $\alpha$-harmonic functions 
which was recently studied by the first author \cite{olofsson2017on}. 
Let $\varphi:\D\to\Hu$ be a bi\-holomorphic map. 
For a function $u$ in $\Hu$ we consider the  weighted pull-back 
\begin{equation}\label{weightedpullback}
v(z)=u_{\varphi,\alpha}(z)=\varphi'(z)^{-\alpha/2} u(\varphi(z)),\quad z\in\D,
\end{equation}
of $u$ by $\varphi$ with respect to the parameter $\alpha$. 
Here 
the power in \eqref{weightedpullback} is defined 
in the usual way using a logarithm of $\varphi'$ in $\D$. 
We shall use the fact that the function $v$ is $\alpha$-harmonic in $\D$ 
if and only if the function $u$ is $\alpha$-harmonic in $\Hu$. 
This latter fact follows easily from \cite[Theorem 1.1]{olofsson2017on}.  
We refer to Geller \cite{Geller} or Ahern et al. \cite{ABC,AC} 
for earlier results.

Let us return to an $\alpha$-harmonic function $u$ in $\Hu$ 
satisfying some appropriate conditions, 
notably \eqref{vanishingDirichletbdrycondition} 
and temperate growth at infinity.
We consider a weighted pull-back $v$ of $u$ of the form 
\eqref{weightedpullback}. The function $v$ is $\alpha$-harmonic in $\D$ 
and we propose to study this function by means of 
its Poisson integral representation $v=P_\alpha[f]$ in $\D$. 
An ambiguity lies in the choice of bi\-holomorphic map 
$\varphi:\D\to\Hu$ which we chose as the M\"obius transformation 
\begin{equation*}
\varphi(z)=i\frac{1+z}{1-z}. 
\end{equation*}  
Notice that $\varphi(0)=i$ and $\varphi(1)=\infty$. 

From \eqref{vanishingDirichletbdrycondition} we have that 
the boundary value $f$ for $v$ vanishes on the set $\T\setminus\{1\}$.  
Standard distribution theory then dictates that 
the distribution $f\in\De'(\T)$ is a finite linear combination 
of derivatives of a Dirac mass $\delta_1$ located at the point $1$ on $\T$: 
$f=\sum_{k=0}^n c_k\delta_1^{(k)}$ in $\De'(\T)$ 
(see H\"ormander \cite[Theorem 2.3.4]{Hormander}). 
The Poisson integral $P_\alpha[\delta_1']$ is naturally interpreted as 
an angular derivative $iAP_\alpha$ of the Poisson kernel $P_\alpha$, and by iteration we find that $v=\sum_{k=0}^n c_k (iA)^kP_\alpha$ (see Corollary \ref{angularderivativePa}).
In Section \ref{sectionangularderivatives}, those angular derivatives $(iA)^k P_\alpha$ of $P_\alpha$ 
are carefully investigated using the M\"obius transformation $\varphi$ above, and in Section \ref{representationthms} 
this part of the analysis culminates in 
the proof of the representation formula \eqref{alphaobstructionfunction} 
(see Theorem \ref{alphaobstructionclass}).

\section{Series expansion of \texorpdfstring{$\alpha$}{alpha}-harmonic functions in \texorpdfstring{$\D$}{D}}\label{sectionhypergeometricfunction}

In this section we revisit the series expansion of $\alpha$-harmonic 
functions in $\D$. Of particular concern is a characterization 
of temperate growth of an $\alpha$-harmonic function in $\D$ 
in terms of polynomial growth of coefficients 
(see Theorem \ref{polynomialgrowthcoefficients}). 
The proof of this latter result depends on properties of 
related hyper\-geometric functions.

For $a\in\C$, we set $(a)_0=1$ and  
$$
(a)_k=\prod_{j=0}^{k-1}(a+j)
$$
for $k=1,2,\dots$. The numbers $(a)_k$ are known as Pochhammer symbols.   
Notice that $(a)_k=\Gamma(a+k)/\Gamma(a)$ for $k\in\N$, 
where $\Gamma$ is the standard Gamma function.

The hyper\-geometric function is the function 
defined by the power series expansion   
\begin{equation}\label{hypergeometricfunction}
F(a,b;c;z)= \sum_{k=0}^\infty \frac{(a)_k(b)_k}{(c)_k} \frac{z^k}{k!}, 
\quad z\in\D,
\end{equation}
for parameters $a,b,c\in\C$ with $c\neq 0,-1,-2,\dots$. 
Convergence in \eqref{hypergeometricfunction} follows by 
the standard ratio test.

Recall also the classical binomial series 
\begin{equation}\label{binomialseries}
(1-z)^{-a}=\sum_{k=0}^\infty \frac{(a)_k}{k!} z^k,
\quad z\in\D,
\end{equation}
for $a\in\C$. Notice that 
$$
F(a,b;b;z)=(1-z)^{-a}, 
$$
which follows from \eqref{hypergeometricfunction} 
and \eqref{binomialseries}.

We take as our starting point a certain series expansion of 
$\alpha$-harmonic functions in $\D$. 
It is known that a function $u$ is $\alpha$-harmonic in $\D$ 
if and only if it has the form 
\begin{equation}\label{generalizedpowerseries}
u(z)=\sum_{k=0}^\infty  c_{k} z^k +  
\sum_{k=1}^\infty  c_{-k}F(-\alpha,k;k+1;\lvert z\rvert^2)\bar z^k,\quad z\in\D,
\end{equation}
for some sequence $\{c_k\}_{k=-\infty}^\infty$ of complex numbers such that 
\begin{equation}\label{coefficientgrowth}
\limsup_{\lvert k\rvert\to\infty}\lvert c_k\rvert^{1/\lvert k\rvert}\leq 1, 
\end{equation}
where $F$ is the hyper\-geometric function \eqref{hypergeometricfunction} 
(see \cite[Theorem 1.2]{olofsson2013poisson}). 
Condition \eqref{coefficientgrowth} 
ensures that the series expansion \eqref{generalizedpowerseries} 
is absolutely  convergent in the space $C^\infty(\D)$ of 
in\-definitely differentiable functions in $\D$. 
As a consequence we have that $u\in C^\infty(\D)$.
We refer to Klintborg and Olofsson \cite{KO} 
for an updated account on these matters. 

The property $F(a,b;c;0)=1$ of the hyper\-geometric function 
\eqref{hypergeometricfunction} 
yields a normalization of the expansion \eqref{generalizedpowerseries}. 
As a consequence we have the coefficient formulas 
$$
c_k=\partial^k u(0)/k!  \quad \text{and}\quad    
c_{-k}=\bar\partial^k u(0)/k!
$$
for $k\in\N$, where $u$ is as in \eqref{generalizedpowerseries}  
(see \cite[Theorem 5.3]{KO}). In particular, 
an $\alpha$-harmonic function in $\D$ is uniquely determined by its 
germ at the origin.

As an example of a series expansion of 
the form \eqref{generalizedpowerseries} we mention that of 
the $\alpha$-harmonic Poisson kernel  
\begin{equation}\label{Paseriesexpansion}
P_\alpha(z)= \sum_{k=0}^\infty z^k +  
\sum_{k=1}^\infty \frac{(\alpha+1)_k}{k!}
F(-\alpha,k;k+1;\lvert z\rvert^2)\bar z^k
\end{equation}
for $z\in\D$ (see \cite[Theorem 6.3]{KO}).

We shall make use of a classical result 
known as Euler's integral formula for the hyper\-geometric function. 
This result says that   
\begin{equation}\label{EulerintegralF}
F(a,b;c;z) = \frac{\Gamma(c)}{\Gamma(b)\Gamma(c-b)} 
\int_0^1 t^{b-1}(1-t)^{c-b-1}(1-tz)^{-a} \, dt 
\end{equation}
for $z\in\D$ when $\re c>\re b>0$,  
where $\Gamma$ is the standard Gamma function 
(see \cite[Theorem 2.2.1]{AAR}).

In particular, from \eqref{EulerintegralF} we have that 
\begin{equation}\label{Ffactorintegralformula}  
F(-\alpha,k;k+1;x)=
k\int_0^1t^{k-1}(1-xt)^\alpha\, dt,\quad 0\leq x<1,
\end{equation}
for $k\in \Z^+$ and $\alpha\in\R$. 
From this latter formula we see that the function $F(-\alpha,k;k+1;\cdot)$ 
is non\-negative on the interval $[0,1)$. Furthermore, 
the function $F(-\alpha,k;k+1;\cdot)$ 
is de\-creasing on the interval $[0,1)$ if $\alpha\geq0$ and, similarly,  
the function $F(-\alpha,k;k+1;\cdot)$ 
is in\-creasing on the interval $[0,1)$ if $\alpha\leq0$.
 
A passage to the limit in \eqref{Ffactorintegralformula} shows that 
\begin{equation}\label{Gausssummationformula}
\lim_{x\to1}  F(-\alpha,k;k+1;x)=
\frac{\Gamma(k+1)\Gamma(\alpha+1)}{\Gamma(k+\alpha+1)}
\end{equation}
for $\alpha>-1$, where we have used a standard formula for 
the Beta function. 
When $\alpha\leq-1$, the quantity $F(-\alpha,k;k+1;x)$ 
diverges to $+\infty$ as $x\to1$.   

We shall need some more detailed estimates of 
the hyper\-geometric functions appearing 
in \eqref{generalizedpowerseries}. 

\begin{lemma}\label{Festimationparameter-1}
Let $k\in\Z^+$. Then 
$$
F(1,k;k+1;x)\leq\frac{k}{x}\log\Big(\frac{1}{1-x}\Big)
$$
for $0<x<1$.
\end{lemma}

\begin{proof}
From \eqref{Ffactorintegralformula} we have that 
\begin{equation*}
F(1,k;k+1;x)=
k\int_0^1 \frac{t^{k-1}}{1-xt}\, dt
\end{equation*}
for $0\leq x<1$. 
An integration by parts shows that 
$$
F(1,k;k+1;x)=\frac{k}{x}\log\Big(\frac{1}{1-x}\Big)
+\frac{k(k-1)}{x}\int_0^1  t^{k-2}\log(1-xt)\, dt
$$ 
for $0< x<1$. Observe that the logarithm in the right\-most 
integral is negative. This yields the conclusion of the lemma.  
\end{proof} 

We shall use also another result of Euler which says that  
\begin{equation}\label{Eulerformula}
F(a,b;c;z)=(1-z)^{c-a-b}F(c-a,c-b;c;z)  
\end{equation}
for $z\in\D$ (see \cite[Theorem 2.2.5]{AAR}).

\begin{lemma}\label{Festimationparameter<-1}
Let $\alpha<-1$ and $k\in\Z^+$. Then 
$$
F(-\alpha,k;k+1;x)\leq\max\Big(1,-\frac{k}{\alpha+1}\Big)(1-x)^{\alpha+1}
$$
for $0\leq x<1$.
\end{lemma}

\begin{proof}
We first apply \eqref{Eulerformula} to see that 
\begin{equation}\label{factorizationformula}
F(-\alpha,k;k+1;x)=(1-x)^{\alpha+1}F(k+\alpha+1,1;k+1;x) 
\end{equation}
for $0\leq x<1$. We shall divide into cases depending on whether 
$k+\alpha+1>0$ or $k+\alpha+1\leq0$.  

Assume first that $k+\alpha+1\leq0$. By \eqref{EulerintegralF} we have that
$$
F(k+\alpha+1,1;k+1;x) =k\int_0^1 (1-t)^{k-1} (1-xt)^{-(k+\alpha+1)}\, dt
$$
for $0\leq x<1$. Since  $k+\alpha+1\leq0$, 
we have that this latter hyper\-geometric function 
$F(k+\alpha+1,1;k+1;\cdot)$ is decreasing on $[0,1)$. 
Therefore $F(k+\alpha+1,1;k+1;x)\leq1$ for $0\leq x<1$. 
In view of \eqref{factorizationformula}  
this yields the conclusion of the lemma 
for $k+\alpha+1\leq0$.  

We next assume that $k+\alpha+1>0$. 
By symmetry and \eqref{EulerintegralF} we have that
\begin{align*}
F(k+\alpha+1,1;k+1;x) &=F(1,k+\alpha+1;k+1;x)\\ 
& =\frac{\Gamma(k+1)}{\Gamma(k+\alpha+1)\Gamma(-\alpha)}
\int_0^1 t^{k+\alpha}(1-t)^{-(\alpha+1)}\frac{1}{1-xt}\, dt 
\end{align*}
for $0\leq x<1$. By monotonicity we have that
\begin{align*}
F(k+\alpha+1,1;k+1;x)&\leq 
\frac{\Gamma(k+1)}{\Gamma(k+\alpha+1)\Gamma(-\alpha)}
\int_0^1 t^{k+\alpha}(1-t)^{-(\alpha+1)-1}\, dt\\ &=   
\frac{\Gamma(k+1) \Gamma(-(\alpha+1))}{\Gamma(-\alpha)\Gamma(k)}
=-\frac{k}{\alpha+1}
\end{align*}
for $0\leq x<1$, where the last two equalities follows by 
standard formulas for the Beta and Gamma  functions.   
In view of \eqref{factorizationformula}  
this yields the conclusion of the lemma 
for $k+\alpha+1>0$.  
\end{proof}

We say that a function $u$ in $\D$ is of temperate growth in $\D$
if it satisfies an estimate of the form
\begin{equation}\label{utemperategrowthdisc}
\lvert u(z)\rvert \leq C(1-\lvert z\rvert^2)^{-N}, \quad z\in\D,
\end{equation}   
for some positive constants $C$ and $N$.

\begin{theorem}\label{polynomialgrowthcoefficients}
Let $\alpha\in\R$. Let $u$ be an $\alpha$-harmonic function in $\D$ and 
consider the expansion \eqref{generalizedpowerseries}. 
Then the function $u$ is of temperate growth in $\D$ if and only if 
the sequence of coefficients $\{c_k\}_{k=-\infty}^\infty$ 
in \eqref{generalizedpowerseries} 
has at most polynomial growth.  
\end{theorem}

\begin{proof}
Assume that $u$ has temperate growth in $\D$. 
We first show that the coefficient $c_k$ has 
at most polynomial growth as $k\to+\infty$. Let $k\in\N$ and $0<r<1$.  
From \eqref{generalizedpowerseries} we have that 
$$
c_k r^k=\frac{1}{2\pi}\int_\T u(re^{i\theta})e^{-ik\theta}\, d\theta.
$$ 
From the triangle inequality and \eqref{utemperategrowthdisc} we have that 
$$
\lvert c_k\rvert r^k\leq C/(1-r^2)^N\leq C/(1-r)^N
$$ 
for $k\in\N$ and $0<r<1$, where $C$ is as in \eqref{utemperategrowthdisc}. 
Choosing $r=1-1/k$ with $k$ big in this latter in\-equality, we see that 
$\lvert c_k\rvert\leq C'k^N$ as $k\to +\infty$.

We next show that the coefficient $c_k$ has 
at most polynomial growth as $k\to-\infty$.  
Let $k\in\Z^-=\Z\setminus\N$ and $0<r<1$. 
From \eqref{generalizedpowerseries} we have that 
\begin{equation}\label{coefficientformula}
c_k F(-\alpha,\lvert k\rvert; \lvert k\rvert+1 ;r^2) r^{\lvert k\rvert}
=\frac{1}{2\pi}\int_\T u(re^{i\theta})e^{-ik\theta}\, d\theta.
\end{equation}
Let us first consider the case $\alpha\leq0$. 
Since the function  
$F(-\alpha,\lvert k\rvert; \lvert k\rvert+1 ;\cdot)$ 
is increasing on $[0,1)$, we have from \eqref{coefficientformula} 
and the triangle inequality that 
$$
\lvert c_k\rvert r^{\lvert k\rvert}\leq C/(1-r)^N,
$$ 
where $C$ is as in \eqref{utemperategrowthdisc}.  
Choosing $r=1-1/\lvert k\rvert$ in this latter inequality we see that 
$\lvert c_k\rvert\leq C'\lvert k\rvert^N$ as $k\to -\infty$.  

Let us now consider the case $\alpha\geq0$. 
Since the function  
$F(-\alpha,\lvert k\rvert; \lvert k\rvert+1 ;\cdot)$ 
is decreasing on $[0,1)$, we have from \eqref{coefficientformula} 
and the triangle inequality that 
\begin{equation}\label{prelcoefficientcontrol}
\lvert c_k\rvert 
\frac{\Gamma(\lvert k\rvert+1)\Gamma(\alpha+1)}{\Gamma(\lvert k\rvert+\alpha+1)}r^{\lvert k\rvert}
\leq C/(1-r)^N,
\end{equation}
where we have used \eqref{Gausssummationformula} and  
$C$ is as in \eqref{utemperategrowthdisc}.
Stirling's formula ensures that 
a quotient of Gamma functions $\Gamma(x)/\Gamma(x+\alpha)$ behaves 
asymptotically as $1/x^\alpha$ when $x\to+\infty$ (\cite[Section 1.4]{AAR}). 
Choosing $r=1-1/\lvert k\rvert$ in \eqref{prelcoefficientcontrol} 
we see that 
$\lvert c_k\rvert\leq C'\lvert k\rvert^{N+\alpha}$ as $k\to -\infty$.  

Assume now that the sequence of coefficients $\{c_k\}_{k=-\infty}^\infty$ 
in \eqref{generalizedpowerseries} has at most polynomial 
polynomial growth, that is, 
$\lvert c_k\rvert\leq C(1+\lvert k\rvert)^N$ for $k\in\Z$, 
where $N\geq0$.  
We shall show that the function $u$ has temperate growth in $\D$.
We consider first the left\-most sum 
$$
f(z)= \sum_{k=0}^\infty  c_{k} z^k, \quad z\in\D,  
$$
appearing in \eqref{generalizedpowerseries}. Observe that 
$$
(1+k)^N\leq \frac{(k+N)!}{k!}= N! \frac{1}{k!} \prod_{j=1}^k(N+j)
= N!\frac{(N+1)_k}{k!}
$$ 
for $k\in\N$.
By the triangle inequality we have that 
$$
\lvert f(z)\rvert\leq C \sum_{k=0}^\infty  (1+k)^N \lvert z\rvert^k
\leq C N! \sum_{k=0}^\infty \frac{(N+1)_k}{k!} \lvert z\rvert^k
= C N! (1-\lvert z\rvert)^{-(N+1)}
$$
for $z\in\D$, where in the last equality we have used \eqref{binomialseries}. 
This proves that the function $f$ above 
has temperate growth in $\D$.

We consider next the function 
$$
g(z)=\sum_{k=1}^\infty  c_{-k}F(-\alpha,k;k+1;\lvert z\rvert^2)\bar z^k,
\quad z\in\D,
$$
appearing in \eqref{generalizedpowerseries}. 
Assume first that $\alpha\geq0$. 
Then the function $F(-\alpha,k;k+1;\cdot)$ is decreasing on $[0,1)$. 
From the triangle inequality we have that
$$
\lvert g(z)\rvert\leq  \sum_{k=1}^\infty  \lvert c_{-k}\rvert \lvert z\rvert^k
\leq C \sum_{k=1}^\infty  (1+k)^N \lvert z\rvert^k
\leq C N! (1-\lvert z\rvert)^{-(N+1)}
$$
for $z\in\D$, where the last two in\-equalities follows as 
in the previous paragraph.
This proves that the function $g$ above 
has temperate growth in $\D$ if $\alpha\geq0$.

Assume next that $-1<\alpha\leq0$. In this case 
the function $F(-\alpha,k;k+1;\cdot)$ is increasing on $[0,1)$. 
By the triangle inequality and \eqref{Gausssummationformula} 
we have that 
$$
\lvert g(z)\rvert\leq \sum_{k=1}^\infty \lvert c_{-k}\rvert   
\frac{\Gamma(k+1)\Gamma(\alpha+1)}{\Gamma(k+\alpha+1)} \lvert z\rvert^k
$$
for $z\in\D$.
Since $\Gamma(x)/\Gamma(x+\alpha)$ behaves 
asymptotically as $1/x^\alpha$ when $x\to+\infty$ we deduce as above that the function $g$ 
has temperate growth in $\D$ if $-1<\alpha\leq0$.   

Assume next that $\alpha=-1$. 
By the triangle inequality and Lemma \ref{Festimationparameter-1}
we have that 
$$
\lvert g(z)\rvert\leq \sum_{k=1}^\infty \lvert c_{-k}\rvert 
\frac{k}{\lvert z\rvert^2}\log\Big(\frac{1}{1-\lvert z\rvert^2}\Big)  
\lvert z\rvert^k\leq 
\frac{C}{\lvert z\rvert^2}\log\Big(\frac{1}{1-\lvert z\rvert^2}\Big)   
\sum_{k=1}^\infty (1+k)^{N+1} \lvert z\rvert^k
$$
for $z\in\D$. 
We can now deduce as above that the function $g$ 
has temperate growth in $\D$ if $\alpha=-1$.   

Assume finally that $\alpha<-1$. 
By the triangle inequality and Lemma \ref{Festimationparameter<-1}
we have that  
$$
\lvert g(z)\rvert\leq \sum_{k=1}^\infty \lvert c_{-k}\rvert 
kC(\alpha) (1-\lvert z\rvert^2)^{\alpha+1} \lvert z\rvert^k\leq 
C C(\alpha)  (1-\lvert z\rvert^2)^{\alpha+1} 
\sum_{k=1}^\infty (1+k)^{N+1} \lvert z\rvert^k
$$
for $z\in\D$, where $C(\alpha)=\max(1,-1/(\alpha+1))$.  
We can then once again deduce as above that the function $g$ 
has temperate growth in $\D$ if $\alpha<-1$.  
\end{proof}

The case division in the proof of 
Theorem \ref{polynomialgrowthcoefficients} depends on different 
behaviors of the hyper\-geometric functions appearing in 
\eqref{generalizedpowerseries}. 
One can notice that the quantity $F(-\alpha,k;k+1;x)$ is 
decreasing in $\alpha\in\R$. 
This observation leads to a shorter proof of 
Theorem \ref{polynomialgrowthcoefficients} 
with less precision. An earlier 
result can be found in Olofsson \cite[Proposition 3.2]{O14}.

\section{Poisson integrals of distributions}\label{sectionpoissonintegrals}

A purpose of this section is to further develop 
a theory of Poisson integral representations of $\alpha$-harmonic 
functions in $\D$. Of particular interest is a characterization 
of Poisson integrals of distributions on $\T$ 
(see Theorem \ref{characterizationPIrepresentation} and 
Corollary \ref{genericcharacterizationPIrepresentation}). 
Along the way we introduce some notation needed later.

We denote by $\De'(\T)$ the space of distributions on $\T$. 
An integrable function $f\in L^1(\T)$ on $\T$ is 
identified with the distribution 
$$
\langle f,\varphi\rangle
=\frac{1}{2\pi}\int_\T f(e^{i\theta})\varphi(e^{i\theta})\, d\theta,
\quad \varphi\in C^\infty(\T),
$$
where $C^\infty(\T)$ is the space of indefinitely differentiable 
test functions on $\T$.
The space $\De'(\T)$ is  topologized in the usual way using the semi-norms 
$$
\De'(\T)\ni f\mapsto \lvert \langle f,\varphi\rangle \rvert 
$$ 
for $\varphi\in C^\infty(\T)$. 
Notice that $f_k\to f$ in $\De'(\T)$  as $k\to\infty$ means that 
the limit 
$\lim_{k\to\infty} \langle f_k,\varphi\rangle =\langle f,\varphi\rangle$ 
holds for every $\varphi\in C^\infty(\T)$. 

Let 
$$
\phi_k(e^{i\theta})=e^{ik\theta},\quad e^{i\theta}\in\T,
$$
for $k\in\Z$ be the exponential monomials on $\T$. 
The Fourier coefficients of a distribution $f\in \De'(\T)$ 
are defined by 
$$
\hat{f}(k)=\langle f,\phi_{-k} \rangle,\quad  k\in\Z.
$$ 
A distribution on $\T$ is uniquely determined by its sequence of 
Fourier coefficients. 
It is well-known that a sequence of complex numbers $\{c_k\}_{k=-\infty}^\infty$  
is of at most polynomial growth if and only if it is 
the sequence of Fourier coefficients for some distribution on $\T$, 
that is, there exists $f\in \De'(\T)$ such that 
$\hat{f}(k)=c_k$ for $k\in\Z$.

The derivative of a distribution $f\in \De'(\T)$ is 
the distribution $f'$ in $\De'(\T)$ determined by 
$(f')\hat{\ }(k)=ik\hat f(k)$ for $k\in\Z$. 

The convolution $h=f*g$ of the distributions $f\in \De'(\T)$ 
and $g\in \De'(\T)$ is the  distribution $h\in \De'(\T)$ determined by 
$\hat{h}(k)=\hat{f}(k)\hat{g}(k)$ for $k\in\Z$. 
Notice that $f*g\in C^\infty(\T)$ if $f\in \De'(\T)$ and $g\in C^\infty(\T)$. 
 
For a suitably smooth function $u$ in $\D$ we set 
\begin{equation}\label{urfcn}
u_r(e^{i\theta})=u(re^{i\theta}),\quad e^{i\theta}\in\T,
\end{equation}
for $0\leq r <1$.   
Notice that the function $u_r$ is essentially the restriction of $u$ to 
the circle $\{z\in\C:\ \lvert z\rvert=r\}$. 

Recall formula \eqref{aPoissonkernel}. 
The $\alpha$-harmonic Poisson integral of a distribution $f\in\De'(\T)$ 
is defined by 
$$
P_\alpha[f](z)=(P_{\alpha,r}\ast f)(e^{i\theta}),\quad z=re^{i\theta}\in\D,
$$
where $\ast$ denotes convolution, 
$P_{\alpha,r}$ is defined in accordance with \eqref{urfcn}, 
$r\geq0$ and $e^{i\theta}\in\T$. 
Observe that 
$$
P_\alpha[f](z)=\frac{1}{2\pi}\int_\T P_{\alpha}(ze^{-i\tau}) f(e^{i\tau})\, d\tau, 
\quad z\in\D,
$$
if $f\in L^1(\T)$. 

We next calculate the series expansion of the Poisson integral.

\begin{prop}\label{PIseriesexpansion} 
Let $\alpha\in\R$ and $f\in\De'(\T)$. Then  
\begin{equation*}
P_\alpha[f](z)= \sum_{k=0}^\infty \hat{f}(k) z^k +  
\sum_{k=1}^\infty \hat{f}(-k)
\frac{(\alpha+1)_k}{k!}
F(-\alpha,k;k+1;\lvert z\rvert^2)\bar z^k
\end{equation*}
for $z\in\D$, where $\hat{f}(k)$ is the $k$-th Fourier coefficient of $f$. 
\end{prop}

\begin{proof} 
Let $0\leq r<1$.
From \eqref{Paseriesexpansion} we have that 
$$
P_{\alpha,r}(e^{i\theta})=  
\sum_{k=0}^\infty r^k e^{ik\theta} +  
\sum_{k=1}^\infty \frac{(\alpha+1)_k}{k!}
F(-\alpha,k;k+1;r^2) r^k e^{-ik\theta}
$$
for $e^{i\theta}\in\T$. Passing to the convolution we have that 
\begin{align*}
P_\alpha[f](z)&=(P_{\alpha,r}\ast f)(e^{i\theta})= 
\sum_{k=0}^\infty r^k\hat{f}(k) e^{ik\theta}\\ & \  +  
\sum_{k=1}^\infty \frac{(\alpha+1)_k}{k!}
F(-\alpha,k;k+1;r^2) r^k \hat{f}(-k) e^{-ik\theta}
\end{align*}
for $z=re^{i\theta}\in\D$ with $r\geq0$ and $e^{i\theta}\in\T$. 
This yields the conclusion of the proposition.
\end{proof}

Notice that the expansion in Proposition \ref{PIseriesexpansion} 
is a series expansion of the form \eqref{generalizedpowerseries}. 
As a consequence, we have that  
the function $P_\alpha[f]$ is $\alpha$-harmonic in $\D$.  
We record also that 
\begin{equation}\label{coefficientbinomialseries}
\frac{(\alpha+1)_k}{k!}=\frac{\Gamma(\alpha+k+1)}{\Gamma(\alpha+1)\Gamma(k+1)}\sim \frac{1}{\Gamma(\alpha+1)}(1+k)^\alpha
\end{equation}
as $k\to+\infty$ which follows by Stirling's formula.

Recall the notion of Fourier spectrum of a distribution on $\T$ 
defined by 
$$
\spec(f)=\{k\in\Z:\ \hat{f}(k)\neq 0\}
$$ 
for $f\in\De'(\T)$.
We shall need a similar notion of spectrum of an $\alpha$-harmonic 
function in $\D$. Recall that an $\alpha$-harmonic function in $\D$ 
is uniquely determined by its 
sequence of coefficients in \eqref{generalizedpowerseries}.    
For such a function $u$ we set 
$$
\spec(u)=\{k\in\Z:\ c_k\neq 0\},
$$ 
where the $c_k$'s are as in \eqref{generalizedpowerseries}. 
Observe that $\spec(u)=\emptyset$ if and only if $u=0$.  
 
Let us denote by $\Z^-$ the set of negative integers. 
In view of the series expansion \eqref{Paseriesexpansion} of 
the function $P_\alpha$ we have that  
$\spec(P_\alpha)=\Z$ if $\alpha\in \R\setminus \Z^-$ whereas 
$$
\spec(P_\alpha)=\{\alpha+1,\ \alpha+2,\dots \}
$$
if $\alpha\in\Z^-$.

\begin{cor}\label{spectraofPI} 
Let $\alpha\in\R$ and $f\in\De'(\T)$. Then 
$$
\spec(P_\alpha[f])=\spec(P_\alpha)\cap\spec(f).
$$
In particular, 
$P_\alpha[f]=0$ if $\alpha\in\Z^-$ and 
$\spec(f)\subset\{k\in\Z:\ k\leq\alpha\}$.
\end{cor}

\begin{proof}
The result is evident from Proposition \ref{PIseriesexpansion}.
When $\alpha\in\Z^-$ 
the inclusion $\spec(f)\subset\{k\in\Z:\ k\leq\alpha\}$ 
ensures that 
$\spec(P_\alpha)\cap\spec(f)=\emptyset$.
\end{proof}

We next turn our attention to a characterization of 
Poisson integrals of distributions.

\begin{theorem}\label{characterizationPIrepresentation}
Let $\alpha\in \R$. 
Then a function $u$ in $\D$ 
has the form of a Poisson integral $u=P_\alpha[f]$ in $\D$ 
for some $f\in\De'(\T)$ if and only if $u$ is $\alpha$-harmonic in $\D$, 
$u$ has temperate growth in $\D$, and $\spec(u)\subset\spec(P_\alpha)$.
\end{theorem}

\begin{proof}
Assume first that  $u=P_\alpha[f]$ in $\D$ for some $f\in\De'(\T)$.  
From a well-known characterization of Fourier coefficients of 
distributions on $\T$ we know that the sequence of 
Fourier coefficients $\{\hat{f}(k)\}_{k=-\infty}^\infty$ has 
at most polynomial growth. 
Proposition \ref{PIseriesexpansion} supplies us with 
the series expansion of $u$. Clearly $u$ is $\alpha$-harmonic in $\D$. 
From Corollary \ref{spectraofPI} we have that  
$\spec(u)\subset\spec(P_\alpha)$.
By Theorem \ref{polynomialgrowthcoefficients} we conclude that 
$u$ has at most temperate growth in $\D$. 

Assume next that $u$ is $\alpha$-harmonic in $\D$, 
$u$ has temperate growth in $\D$, and $\spec(u)\subset\spec(P_\alpha)$.
Consider the series expansion \eqref{generalizedpowerseries}. 
By Theorem \ref{polynomialgrowthcoefficients} we conclude that 
the sequence of coefficients $\{c_k\}_{k=-\infty}^\infty$ 
in \eqref{generalizedpowerseries}
has at most polynomial growth. 
We set $a_k=c_k$ for $k\in\N$, 
$$
a_k=\frac{(-k)!}{(\alpha+1)_{-k}}c_{k}
$$ 
for $k\in\Z^-$ and $k\in \spec(P_\alpha)$, 
and $a_k=0$ for $k\in\Z^-$ and $k\not\in \spec(P_\alpha)$. 
From \eqref{coefficientbinomialseries} we have that 
the sequence $\{a_k\}_{k=-\infty}^\infty$ is of polynomial growth.
From a well-known characterization of Fourier coefficients of 
distributions on $\T$ there exists $f\in\De'(\T)$ such that 
$\hat{f}(k)=a_k$ for $k\in\Z$.
Since $\spec(u)\subset\spec(P_\alpha)$, 
we have by Proposition \ref{PIseriesexpansion} that 
$u=P_\alpha[f]$ in $\D$.     
\end{proof}

The spectral condition in Theorem \ref{characterizationPIrepresentation}  
is redundant when 
$\alpha$ is not a negative integer.

\begin{cor}\label{genericcharacterizationPIrepresentation}
Let $\alpha\in \R\setminus\Z^-$. 
Then a function $u$ in $\D$ 
has the form of a Poisson integral $u=P_\alpha[f]$ in $\D$ 
for some $f\in\De'(\T)$ if and only if it is $\alpha$-harmonic in $\D$ 
and of temperate growth there.
\end{cor}

\begin{proof}
Recall that $\spec(P_\alpha)=\Z$ in the present situation. 
The result follows from Theorem \ref{characterizationPIrepresentation}.
\end{proof}

The $\alpha$-harmonic Poisson kernel $P_\alpha$ has bounded $L^1$-means 
when $\alpha>-1$ (see Olofsson \cite[Theorem 3.1]{O14}). 
As a consequence, for $\alpha>-1$, 
one has that $u_r\to f$ in $\De'(\T)$ as $r\to1$ 
if $f\in\De'(\T)$ and $u=P_\alpha[f]$, 
where $u_r$ is as in \eqref{urfcn}.

The boundary behavior of $\alpha$-harmonic functions in $\D$ 
is conceptually different when  $\alpha\leq-1$. 
The next result exemplifies this fact.  

\begin{theorem}\label{boundarybehavioralphaleq-1}
Let $\alpha\leq-1$ and let $u$ be an $\alpha$-harmonic function in $\D$. 
Assume that the limit $f=\lim_{r\to1}u_r$ in $\De'(\T)$ exists, 
where $u_r$ is as in \eqref{urfcn}. 
Then $u$ is analytic in $\D$.
\end{theorem}

\begin{proof}
Consider the series expansion \eqref{generalizedpowerseries}. 
We shall prove that $c_k=0$ for $k<0$. Let $k\in\Z^+$. 
From \eqref{generalizedpowerseries} we have that 
\begin{equation}\label{ccoefficientformula}
c_{-k} F(-\alpha,k; k+1 ;r^2) r^{k}
=\frac{1}{2\pi}\int_\T u(re^{i\theta})e^{ik\theta}\, d\theta
\end{equation}
for $0<r<1$. 
By assumption the right hand side in \eqref{ccoefficientformula} has 
a limit as $r\to 1$. 
Recall from the discussion following 
formulas \eqref{Ffactorintegralformula}-\eqref{Gausssummationformula}  
that the quantity $F(-\alpha,k; k+1 ;x)$ increases to $+\infty$ as $x\to1$.
Passing to the limit in \eqref{ccoefficientformula} as $r\to1$ 
we conclude that $c_{-k}=0$. This yields the conclusion of the theorem.
\end{proof}

We refer to Olofsson \cite[Theorem 2.3]{O14} for a result analogous to 
Theorem \ref{boundarybehavioralphaleq-1}
phrased in another setting of generalized harmonic functions in $\D$.  
Under an assumption of pointwise boundary limit, 
a similar result also exists for the class of 
generalized axially symmetric potentials mentioned in the introduction, 
see Huber \cite[Theorem 2]{Huber}.

A traditional approach to the Poisson integral representation
\begin{equation}\label{PIrepresentation}
u(z)=P_\alpha[f](z), \quad z\in\D, 
\end{equation}
is to exhibit the distribution $f\in\De'(\T)$ 
as a limit point of the $u_r$'s as $r\to1$ 
and then deduce the Poisson integral representation \eqref{PIrepresentation} 
from a uniqueness argument. 
See for instance 
Olofsson and Wittsten \cite[Theorem 5.5]{olofsson2013poisson} 
for an elaboration on this theme. An interesting point of 
Theorem \ref{characterizationPIrepresentation} is that this result 
establishes  Poisson integral representations \eqref{PIrepresentation} 
in cases where boundary limits are non-existent.

Structure in \eqref{generalizedpowerseries} suggests use of 
the differential operator $A=z\partial-\bar z\bar\partial$. 
Observe that 
\begin{equation}\label{Aaction}
Au(z)= \sum_{k=1}^\infty  kc_{k} z^k 
-\sum_{k=1}^\infty  k c_{-k}F(-\alpha,k;k+1;\lvert z\rvert^2)\bar z^k
,\quad z\in\D,
\end{equation}
if $u$ has the form \eqref{generalizedpowerseries}. 
In particular, the function $Au$ is $\alpha$-harmonic in $\D$ if $u$ is.

\begin{prop}\label{A'intertwining}
Let $\alpha\in\R$ and $f\in\De'(\T)$. Then
$$
iA P_\alpha[f](z)=P_\alpha[f'](z)
$$
for $z\in\D$, where $f'$ is the distributional derivative of $f$.
\end{prop}

\begin{proof} 
Recall that $(f')\hat{\ }(k)=ik\hat f(k)$ for $k\in\Z$. 
By Proposition \ref{PIseriesexpansion} we have that 
\begin{align*}
P_\alpha[f'](z)&=\sum_{k=0}^\infty ik\hat{f}(k) z^k +  
\sum_{k=1}^\infty(-ik) \hat{f}(-k) \frac{(\alpha+1)_k}{k!}
F(-\alpha,k;k+1;\lvert z\rvert^2)\bar z^k\\
&=iAP_\alpha[f](z)
\end{align*}
for $z\in\D$, where in the last equality we have used \eqref{Aaction}.
\end{proof}

We refer to the differential operator 
$iA=i(z\partial-\bar z\bar\partial)$ as the angular derivative.
Our interest in this operator arose in connection to 
the paper Olofsson \cite{olofsson2018lipschitz}. 

Let $\delta_1\in\De'(\T)$ be the unit Dirac mass located at 
the point $1\in\T$, that is, 
$\langle\delta_1,\varphi\rangle=\varphi(1)$ for $\varphi\in C^\infty(\T)$.

\begin{cor}\label{angularderivativePa}
Let $\alpha\in\R$ and $k\in\N$. Then
$$
P_\alpha[\delta_1^{(k)}](z)=(iA)^k P_\alpha(z)
$$
for $z\in\D$, 
where $\delta_1^{(k)}$ denotes the $k$-th distributional derivative 
of $\delta_1$.
\end{cor}

\begin{proof}
It is straight\-forward to check that $(\delta_1)\hat{\ }(k)=1$ for $k\in\Z$.  
By Proposition \ref{PIseriesexpansion} and formula \eqref{Paseriesexpansion} 
we have that 
$P_\alpha[\delta_1](z)=P_\alpha$.   
The result now follows by Proposition \ref{A'intertwining}. 
\end{proof}

\section{Angular derivatives of Poisson kernels}
\label{sectionangularderivatives}

This section is devoted to a careful analysis of 
angular derivatives of Poisson kernels.    
We shall make good use of the M\"obius transformation 
\begin{equation}\label{mobiusmap}
\varphi(z)=i\frac{1+z}{1-z}
\end{equation} 
in our calculations. 
Notice that $\varphi$ maps $\D$ one-to-one onto $\Hu$, $\varphi(0)=i$ 
and $\varphi(1)=\infty$. 
From standard theory we have that the M\"obius transformation $\varphi$ 
is uniquely determined by these three properties. 
For the sake of easy reference we record also the 
formulas  
\begin{equation}\label{mobiusmapderivative}
\im \varphi(z)=\frac{1-\lvert z\rvert^2}{\lvert 1-z\rvert^2} 
\quad\text{and}\quad \varphi'(z)=\frac{2i}{(1-z)^2},
\end{equation} 
which are straight\-forward to check.

A most natural $\alpha$-harmonic function in $\Hu$ is the $(\alpha+1)$-th power 
of the imaginary part:
$$
u(z)=(\im z)^{\alpha+1},\quad z\in\Hu.
$$
We first calculate the weighted pull-back $u_{\varphi,\alpha}$ 
from \eqref{weightedpullback} of this function by $\varphi$.

\begin{theorem}\label{Paaswpullback}
Let $\alpha\in\R$. Then 
\begin{equation*}
P_\alpha(z)=
c\varphi'(z)^{-\alpha/2} (\im \varphi(z))^{\alpha+1},
\quad z\in\D,
\end{equation*}
where $\varphi$ is as in \eqref{mobiusmap} and $c\varphi'(0)^{-\alpha/2}=1$.  
\end{theorem}

\begin{proof}
Recall formula \eqref{aPoissonkernel}. 
The two formulas in \eqref{mobiusmapderivative}   make evident that 
$$
P_\alpha(z)=c\varphi'(z)^{-\alpha/2} (\im \varphi(z))^{\alpha+1},
\quad z\in\D,
$$
where $c\in\C$. 
From $P_\alpha(0)=1$ we see that $c\varphi'(0)^{-\alpha/2}=1$.  
\end{proof}

Notice that Theorem \ref{Paaswpullback} with $\alpha=0$ yields 
the well-known 
formula 
$$
P_0(z)=\re\Big( \frac{1+z}{1-z} \Big),\quad z\in\D,
$$
for the usual 
Poisson kernel for $\D$.

Our next task is to calculate angular derivatives of Poisson kernels.  
We begin with a preparatory lemma.  

\begin{lemma}\label{iAmobiusmap} 
Let $\varphi$ be as in \eqref{mobiusmap}. Then $iA\varphi=\frac{1}{2}(\varphi^2+1)$.
\end{lemma}

\begin{proof}
From formula \eqref{mobiusmapderivative} we have that
$$
iA\varphi(z)=iz\varphi'(z)=-2z/(1-z)^2.
$$ 
Using some elementary algebra we now calculate that  
$$
iA\varphi(z)=\frac{1}{2}\frac{(1-z)^2-(1+z)^2}{(1-z)^2}
=\frac{1}{2}\Big(1- \Big(\frac{1+z}{1-z}\Big)^2 \Big)
=\frac{1}{2}(1+\varphi(z)^2),
$$ 
where the last equality follows by  \eqref{mobiusmap}.
\end{proof}

Let us record some properties of the angular derivative. 
The differential operator $iA$ satisfies the product rule 
for differen\-tiation: $iA(fg)=giA(f)+fiA(g)$ for suitable $f$ and $g$. 
Denote by $\bar f$ the point\-wise complex conjugate of 
a complex-valued function $f$. 
We shall use the chain rule in the form 
\begin{equation}\label{chainrule}
iA(h\circ f)=((\partial h)\circ f)iAf
+((\bar\partial h)\circ f)\overline{iAf}
\end{equation}
for suitable $f$ and $h$. 
Formula \eqref{chainrule} is straight\-forward to check.
As a consequence of \eqref{chainrule} we have that 
the differential operator $iA$ commutes 
with the action of complex conjugation of functions: 
$iA(\bar f)=\overline{iAf}$ for suitable $f$.

We now turn to differentiation of powers.

\begin{lemma}\label{iAimmobiusmappower} 
Let $\varphi$ be as in \eqref{mobiusmap} and $\alpha\in\R$. 
Then $iA((\im \varphi)^{\alpha+1})
=(\alpha+1)(\re \varphi)
(\im \varphi)^{\alpha+1}$ in $\D$.
\end{lemma}

\begin{proof}
Let $u=(\im \varphi)^{\alpha+1}$ in $\D$.  
By a standard rule for differentiation we have that
$$
iAu=(\alpha+1)(\im \varphi)^{\alpha} iA(\frac{1}{2i}(\varphi-\bar\varphi))
$$
in $\D$, compare with \eqref{chainrule}.
Since the differential operator $iA$ commutes with the action of 
complex conjugation of functions we have that 
$$
iAu=(\alpha+1)(\im \varphi)^{\alpha} 
\frac{1}{2i}(iA\varphi-\overline{iA\varphi})
$$
in $\D$.   
We now use Lemma \ref{iAmobiusmap} and calculate that  
$$
iAu=(\alpha+1)(\im \varphi)^{\alpha} \frac{1}{2i}
(\frac{1}{2}(\varphi^2+1)-\frac{1}{2}(\bar\varphi^2+1))=  
(\alpha+1)(\im \varphi)^{\alpha} 
\frac{1}{2i}\frac{1}{2}(\varphi^2-\bar\varphi^2)
$$
in $\D$, 
where the last equality follows by cancellation. 
By some elementary algebra we now have that  
$$
iAu=(\alpha+1)(\im \varphi)^{\alpha} \frac{1}{2i}\frac{1}{2}
(\varphi-\bar\varphi)(\varphi+\bar\varphi)=
(\alpha+1)(\im \varphi)^{\alpha}(\im \varphi)(\re \varphi)
$$ 
in $\D$. This yields the conclusion of the lemma.
\end{proof}

\begin{lemma}\label{iAmobiusmappower} 
Let $\varphi$ be as in \eqref{mobiusmap} and $\alpha\in\R$. 
Then $iA((\varphi')^{-\alpha/2})
=\frac{\alpha}{2} (i-\varphi)(\varphi')^{-\alpha/2}$ 
in $\D$. 
\end{lemma}

\begin{proof}
We first show that $iA\varphi'=(\varphi-i)\varphi'$. 
From \eqref{mobiusmapderivative} we have that 
$$
iA\varphi'(z)=iz\varphi''(z)=(2i)^2\frac{z}{(1-z)^3}
=i\frac{2z}{1-z}\varphi'(z),
$$ 
where the last equality again follows by \eqref{mobiusmapderivative}.
By some elementary algebra we now have that  
$$
iA\varphi'(z)=
i\frac{(1+z)-(1-z)}{1-z}\varphi'(z)
=i\Big(\frac{1+z}{1-z}-1\Big)\varphi'(z)=(\varphi(z)-i)\varphi'(z),
$$
where the last equality follows by \eqref{mobiusmap}.

We now consider the general case. 
By a well-known differentiation formula we have that 
$$
iA((\varphi')^{-\alpha/2})=-\frac{\alpha}{2}(\varphi')^{-\alpha/2-1}iA\varphi' 
$$
in $\D$, compare with \eqref{chainrule}. 
We now use the result of the previous paragraph to conclude that 
$$
iA((\varphi')^{-\alpha/2})=-\frac{\alpha}{2}(\varphi')^{-\alpha/2-1}    
(\varphi-i)\varphi'=\frac{\alpha}{2}(i-\varphi)(\varphi')^{-\alpha/2}
$$
in $\D$. 
\end{proof}

We next calculate the angular derivative of the Poisson kernel. 

\begin{theorem}\label{iAderivativePa}
Let $\alpha\in\R$ and let $\varphi$ be as in \eqref{mobiusmap}. Then 
$$
iA P_\alpha(z)=\frac{1}{2}(\varphi(z)+(\alpha+1)\overline{\varphi(z)}
+i\alpha )P_\alpha(z)  
$$
for $z\in\D$. 
\end{theorem}

\begin{proof}
Let $u=(\varphi')^{-\alpha/2} (\im \varphi)^{\alpha+1}$ in $\D$. 
In view of Theorem \ref{Paaswpullback} it suffices to prove  
the conclusion of the theorem with $P_\alpha$ replaced 
by the function $u$. 
By the product rule for differentiation we have that 
$$
iAu= (\im \varphi)^{\alpha+1} iA((\varphi')^{-\alpha/2})      
+ (\varphi')^{-\alpha/2} iA((\im \varphi)^{\alpha+1})
$$
in $\D$. 
We now use Lemmas \ref{iAimmobiusmappower} and \ref{iAmobiusmappower} 
to conclude that 
\begin{align*}
iAu&=(\im \varphi)^{\alpha+1} \frac{\alpha}{2} (i-\varphi)(\varphi')^{-\alpha/2}
+ (\varphi')^{-\alpha/2} 
(\alpha+1)(\re \varphi)
(\im \varphi)^{\alpha+1}\\
&= \Big(\frac{\alpha}{2}(i-\varphi) +
(\alpha+1)
\frac{1}{2}(\varphi+\bar\varphi)   
\Big) u
\end{align*}
in $\D$, where the last equality is straight\-forward to check. 
This yields the conclusion of the theorem.
\end{proof}

An earlier version of Theorem \ref{iAderivativePa} appears in 
Olofsson \cite[Theorem 1.11]{olofsson2018lipschitz};  
see also Klintborg and Olofsson \cite[Corollary 1.3]{KO}
for a generalization along similar lines. 
Here our focus is on applications to the setting of 
the upper half-plane and we express matters using the M\"obius 
transformation $\varphi$.

We denote by $\C[z,\bar z]$ the algebra of polynomials in $z$ and $\bar z$ 
with complex coefficients.
Define polynomials $\{h_{k,\alpha}\}_{k=0}^\infty$ in $\C[z,\bar z]$ 
by $h_{0,\alpha}=1$ and 
\begin{align}\label{hkpolynomials}
h_{k+1,\alpha}(z)&=\frac{1}{2}(z^2+1)\partial h_{k,\alpha}(z)+
\frac{1}{2}(\bar z^2+1)\bar\partial h_{k,\alpha}(z)\\ 
& \quad +\frac{1}{2}(z+(\alpha+1)\bar z+i\alpha) h_{k,\alpha}(z) \notag
\end{align}
for $k\geq0$. It is straight\-forward to check that $h_{k,\alpha}$ 
has degree at most $k$ for $k=0,1,\dots$.

\begin{theorem}\label{iApowerPa} 
Let $\alpha\in\R$ and let $\varphi$ be as in \eqref{mobiusmap}. Then 
$$
(iA)^k P_\alpha(z)=h_{k,\alpha}(\varphi(z))P_\alpha(z),\quad z\in\D,  
$$
for $k=0,1,2,\dots$, 
where the $h_{k,\alpha}$'s are as in \eqref{hkpolynomials}. 
\end{theorem}

\begin{proof}
For $k=0$ the result is evident. 
For $k=1$ the result follows by Theorem \ref{iAderivativePa}. 
We proceed by induction and assume that  
$(iA)^k P_\alpha=(h_{k,\alpha}\circ\varphi)P_\alpha$
for some $k\geq0$. Applying the operator $iA$ we have that 
$$
(iA)^{k+1} P_\alpha= P_\alpha iA(h_{k,\alpha}\circ\varphi)
+(h_{k,\alpha}\circ\varphi) iAP_\alpha,  
$$
where we have used the product rule. 
From the chain rule \eqref{chainrule} and Lemma \ref{iAmobiusmap} 
we have that 
$$
iA(h_{k,\alpha}\circ\varphi) =((\partial h_{k,\alpha})\circ\varphi)
\frac{1}{2}(\varphi^2+1)
+((\bar\partial h_{k,\alpha})\circ\varphi)\frac{1}{2}(\bar\varphi^2+1). 
$$
We now return to the function $(iA)^{k+1} P_\alpha$ and use 
Theorem \ref{iAderivativePa} to conclude that 
\begin{align*}
(iA)^{k+1} P_\alpha&= P_\alpha\Big(  ((\partial h_{k,\alpha})\circ\varphi)
\frac{1}{2}(\varphi^2+1)
+((\bar\partial h_{k,\alpha})\circ\varphi)\frac{1}{2}(\bar\varphi^2+1)\Big)\\ 
& \ + (h_{k,\alpha}\circ\varphi) 
\frac{1}{2}(\varphi+(\alpha+1)\bar\varphi+i\alpha )P_\alpha 
= (h_{k+1,\alpha}\circ\varphi) P_\alpha,
\end{align*}
where in the last equality we have used \eqref{hkpolynomials}.
The result now follows by invoking the induction principle.
\end{proof}

We shall study the polynomials $h_{k,\alpha}$ in some more detail.

\begin{lemma}\label{dehka}
Let $h_{k,\alpha}\in\C[z,\bar z]$ be as in \eqref{hkpolynomials} 
for some $\alpha\in\R$ and $k\geq0$. Then the function  
$\Hu\ni z\mapsto (\im z)^{\alpha+1}h_{k,\alpha}(z)$  
is $\alpha$-harmonic in $\Hu$.   
\end{lemma}

\begin{proof}
Recall that the function $iAu$ is $\alpha$-harmonic in $\D$ if $u$ is, see \eqref{Aaction}.
From Theorem \ref{iApowerPa} we thus have that the function 
$(h_{k,\alpha}\circ\varphi)P_\alpha$ is $\alpha$-harmonic in $\D$. 
Theorem \ref{Paaswpullback} displays 
the Poisson kernel $P_\alpha$ as a 
weighed pull-back by the function $\varphi$.   
From Olofsson \cite[Theorem 1.1]{olofsson2017on} 
we now conclude that the function 
$z\mapsto (\im z)^{\alpha+1}h_{k,\alpha}(z)$  
is $\alpha$-harmonic in $\Hu$, 
see discussion following formula \eqref{weightedpullback}.   
\end{proof}

Let $\alpha\in\R$ and let us denote by $y$ the imaginary part. 
A calculation shows that 
$$
\Delta_{\alpha;\Hu}y^{\alpha+1}=\frac{1}{2i}D_\alpha
$$
in the sense of differential operators, where 
\begin{equation}\label{Dalphaoperator}
D_{\alpha,z}=(z-\bar z)\partial_z\bar\partial_z
+\bar\partial_z-(\alpha+1)\partial_z.
\end{equation} 
As a consequence, we have that a product $y^{\alpha+1}h$ is 
$\alpha$-harmonic in $\Hu$ if and only  if $D_\alpha h=0$ in $\Hu$, 
where $D_\alpha$ is as in \eqref{Dalphaoperator}.
   
Notice that if $p\in\C[z,\bar z]$ is homogeneous of degree $m$, 
then the polynomial $D_\alpha p$ is homo\-geneous of degree $m-1$.

Consider the partial sums
\begin{equation}\label{eq:skallealfa}
s_{k,\alpha}(z)=\sum_{j=0}^k \frac{(\alpha+1)_j}{j!} z^j
\end{equation}
for $k=0,1,\dots$ of a binomial series \eqref{binomialseries} 
with $a=\alpha+1$.
We shall need the associated homo\-geneous polynomials 
\begin{equation}\label{eq:pkallealfa}
p_{k,\alpha}(z)=z^k s_{k,\alpha}(\bar z /z)
=\sum_{j=0}^k\frac{(\alpha+1)_j}{j!}z^{k-j}\bar z^j
\end{equation}
for $k=0,1,\dots$, where $\alpha\in\R$.
Notice that $p_{k,\alpha}\in\C[z,\bar z]$ is homo\-geneous of degree $k$.

We now return to the differential operator $D_\alpha$ in \eqref{Dalphaoperator}.

\begin{theorem}\label{significancepkapolynomial}
Let $\alpha\in\R$. 
Let $p$ in $\C[z,\bar z]$ be  homo\-geneous of degree $k$. 
Then $D_\alpha p=0$ if and only if $p$ is a constant multiple of $p_{k,\alpha}$, 
where $p_{k,\alpha}$ is as in \eqref{eq:pkallealfa}. 
\end{theorem}

\begin{proof}
We shall evaluate the differential operator $D_\alpha$ 
on a polynomial $p$ of the form 
\begin{equation}\label{homogeneouspolynomial}
p(z)=\sum_{j=0}^k a_j z^{k-j}\bar z^j
\end{equation}
for some $a_0,\dots,a_k\in\C$. Calculations show that 
$$
\bar\partial p(z)=\sum_{j=0}^{k-1} (j+1)a_{j+1} z^{k-1-j}\bar z^j,\quad   
\partial p(z)=\sum_{j=0}^{k-1} (k-j)a_{j} z^{k-1-j}\bar z^j,
$$
and 
$$
(z-\bar z)\partial\bar\partial p(z)=   
\sum_{j=0}^{k-2} (j+1)(k-1-j)a_{j+1} z^{k-1-j}\bar z^j-
\sum_{j=1}^{k-1} j(k-j)a_{j} z^{k-1-j}\bar z^j.
$$
From these formulas we have that 
$$
D_\alpha p(z)=    
\sum_{j=0}^{k-1} (k-j)((j+1)a_{j+1}-(j+\alpha+1)a_{j}) z^{k-1-j}\bar z^j
$$
for $p$ of the form \eqref{homogeneouspolynomial}. 

From the result of the previous paragraph we have that $D_\alpha p=0$ 
if and only if 
$$
(j+1)a_{j+1}-(j+\alpha+1)a_{j}=0
$$  
for $j=0,\dots, k-1$. We conclude that $D_\alpha p=0$ 
if and only if $p=a_0 p_{k,\alpha}$. 
\end{proof}

\begin{cor}\label{reprtermharmonic}
Let $\alpha\in\R$ and $k\in\N$. 
Let $p_{k,\alpha}$ be as in \eqref{eq:pkallealfa}. Then the function  
$\Hu\ni z\mapsto (\im z)^{\alpha+1}p_{k,\alpha}(z)$  
is $\alpha$-harmonic in $\Hu$.   
\end{cor}

\begin{proof}
By Theorem \ref{significancepkapolynomial} we have that 
$D_\alpha p_{k,\alpha}=0$. 
Therefore $\Delta_\alpha y^{\alpha+1}p_{k,\alpha}=0$ in $\Hu$ by 
the factorization formula preceding \eqref{Dalphaoperator}.   
\end{proof}

We now return to the $h_{k,\alpha}$'s.

\begin{theorem}\label{hkalincombofpka}
Let $\alpha\in\R$ and $k\in\N$. 
Let $h_{k,\alpha}$ be as in \eqref{hkpolynomials}. 
Then $h_{k,\alpha}$ is a linear combination of 
$p_{0,\alpha},\dots,p_{k,\alpha}$,  
where the $p_{j,\alpha}$'s are as in \eqref{eq:pkallealfa}. 
\end{theorem}

\begin{proof}
From Lemma \ref{dehka} we have that 
the product $y^{\alpha+1}h_{k,\alpha}$ is $\alpha$-harmonic in $\Hu$. 
Therefore $D_\alpha h_{k,\alpha}=0$ in $\C[z,\bar z]$, 
where $D_\alpha$ is as in \eqref{Dalphaoperator}.  
We next write the polynomial $h_{k,\alpha}$ as a sum of 
homo\-geneous polynomials: 
$h_{k,\alpha}=\sum_{j=0}^k p_j$, where $p_j$ is homogeneous of degree $j$. 
Clearly $p_0$ is a constant multiple of $p_{0,\alpha}$.
Applying the differential operator $D_\alpha$ we see that 
$$
0=D_\alpha h_{k,\alpha}=\sum_{j=1}^k D_\alpha p_j
$$ 
in $\C[z,\bar z]$. Since $D_\alpha p_j$ is homogeneous of degree $j-1$,
we have that $D_\alpha p_j=0$ for $j=1,\dots k$.  
From Theorem \ref{significancepkapolynomial} 
we have that $p_j$ is a constant multiple of $p_{j,\alpha}$ 
for $j=1,\dots k$. 
We conclude that $h_{k,\alpha}$ is a linear combination of 
$p_{0,\alpha},\dots,p_{k,\alpha}$.
\end{proof}

\section{A general representation theorem}
\label{representationthms}

In this section we consider $\alpha$-harmonic functions in $\Hu$ 
which are of temperate growth at infinity and vanish on the real line. 
A main task is to establish a re\-presentation of 
such functions. 
The analysis uses results from Section \ref{sectionangularderivatives}.

\begin{theorem}\label{alphaobstructionclass}
Let $\alpha>-1$.  
Let $u$ be an $\alpha$-harmonic function in $\Hu$ which is 
of temperate growth at infinity. 
Assume that $u(z)\to0$ as $\Hu\ni z\to x$ for every $x\in\R$.  
Then 
\begin{equation}\label{alphaobstructionfcn}
u(z)=\sum_{k=0}^n  c_k(\im z)^{\alpha+1} p_{k,\alpha}(z),\quad z\in\Hu,
\end{equation}
for some $n\in\N$ and $c_0,\dots,c_n\in\C$, 
where the $p_{k,\alpha}$'s are as in \eqref{eq:pkallealfa}.
\end{theorem}

\begin{proof}
We consider the weighted pull-back 
$$
v(z)=\varphi'(z)^{-\alpha/2}u(\varphi(z)),\quad z\in\D,
$$
where $\varphi$ is as in \eqref{mobiusmap}.  
From \cite[Theorem 1.1]{olofsson2017on} it follows that $v$ is $\alpha$-harmonic in $\D$.
From temperate growth of $u$ at infinity and vanishing of $u$ on $\R$
we have by a compactness argument that $v$ is of temperate growth in $\D$.
By Theorem \ref{characterizationPIrepresentation} we conclude that 
$v=P_\alpha[g]$ for some distribution $g\in\De'(\T)$. 
Since $\alpha>-1$, we have from general theory 
that $v_r\to g$ in $\De'(\T)$ as $r\to1$, 
where the $v_r$'s are defined as in \eqref{urfcn} 
(see for instance \cite[Theorem 5.4]{olofsson2013poisson}).
Since $u(z)\to0$ as $\Hu\ni z\to x$ for $x\in\R$,  
we see that $g$ vanishes on $\T\setminus\{1\}$ in a distributional sense.
Now since $\supp(g)\subset\{1\}$,
standard distribution theory dictates that 
$g=\sum_{k=0}^na_k\delta_1^{(k)}$ in $\De'(\T)$ 
for some $n\in\N$ and complex numbers $a_0,\dots, a_n\in\C$ 
(see H\"ormander \cite[Theorem 2.3.4]{Hormander}). 
By Corollary \ref{angularderivativePa} we have that 
$P_\alpha[\delta_1^{(k)}]=(iA)^kP_\alpha$ for $k\ge0$.
Passing to the Poisson integral we have that 
\begin{equation}\label{vprelformula}
v(z)= \sum_{k=0}^n a_k (iA)^kP_\alpha(z)
=\Big(\sum_{k=0}^n a_k h_{k,\alpha}(\varphi(z)) \Big) P_\alpha(z)
\end{equation}
for $z\in\D$, where the last equality follows by Theorem \ref{iApowerPa}. 
By Theorem \ref{hkalincombofpka} the function  
$h_{k,\alpha}$ is  a linear combination of 
$p_{0,\alpha},\dots,p_{k,\alpha}$. From \eqref{vprelformula} we conclude that 
\begin{equation*}
v(z)=\Big(\sum_{k=0}^n b_k p_{k,\alpha}(\varphi(z)) \Big) P_\alpha(z),\quad z\in\D,
\end{equation*}
for some $b_0,\dots, b_n\in\C$.   
Invoking Theorem \ref{Paaswpullback} we see that 
\begin{equation*}
v(z)=\varphi'(z)^{-\alpha/2} 
\Big(\sum_{k=0}^n cb_k (\im\varphi(z))^{\alpha+1} p_{k,\alpha}(\varphi(z)) \Big)
\end{equation*}
for $z\in\D$, where $c\varphi'(0)^{-\alpha/2}=1$.
A passage back to the function $u$ yields 
\eqref{alphaobstructionfcn} with $c_k=cb_k$ for $k=0,\dots,n$. 
This completes the proof of the theorem. 
\end{proof}

For $\alpha>-1$ and $n\in\N$, 
we denote by $\mathcal{V}_{\alpha,n}$ the set of all functions $u$ 
of the form \eqref{alphaobstructionfcn} for some $c_0,\dots,c_n\in\C$, 
where the $p_{k,\alpha}$'s are as in \eqref{eq:pkallealfa}. 
We also set 
$$
\mathcal{V}_{\alpha}=\cup_{n=0}^\infty\mathcal{V}_{\alpha,n}.
$$
The space $\mathcal{V}_{\alpha,n}$ is a complex vector space 
of finite dimension $n+1$. 
Notice that the set $\mathcal{V}_{\alpha}$ is a complex vector space 
which is naturally 
filtered by the sets $\mathcal{V}_{\alpha,n}$ for $n=0,1,\dots$.

Observe that the conclusion of Theorem \ref{alphaobstructionclass} 
says that $u\in\mathcal{V}_{\alpha}$. 
We next describe the class $\mathcal{V}_{0}$ in some more detail.

\begin{prop}\label{harmonicobstructionclass}
Let $n\in\N$. Then a function $u$ belongs to the space $\mathcal{V}_{0,n}$ 
if and only if it has the form  
\begin{equation}\label{harmonicobstructionfcn}   
u(z)=\sum_{k=1}^{n+1}  c_k \im(z^k),\quad z\in\Hu,
\end{equation}
for some $c_1,\dots,c_{n+1}\in\C$.
\end{prop}

\begin{proof}  
A calculation using the formula for a finite geometric sum shows that  
$$
p_{k,0}(z)=\sum_{j=0}^k z^{k-j}\bar z^j=\frac{z^{k+1}-\bar z^{k+1} }{z-\bar z}.
$$
Therefore 
$$
(\im z)p_{k,0}(z)=\frac{z-\bar z}{2i} \frac{z^{k+1}-\bar z^{k+1}}{z-\bar z} =\im(z^{k+1}).
$$
The result is now evident from definition of the space $\mathcal{V}_{0,n}$. 
\end{proof}

We shall next investigate the order of growth of 
a function 
of the form   
\begin{equation}\label{ukafunction}
u_{k,\alpha}(z)=   (\im z)^{\alpha+1} p_{k,\alpha}(z),\quad z\in\Hu,
\end{equation}
for some $k\in\N$ and $\alpha>-1$, 
where $p_{k,\alpha}$ is as in \eqref{eq:pkallealfa}.  
From the proof of Proposition \ref{harmonicobstructionclass} 
we have that
$$
u_{k,0}(z)= \im(z^{k+1}) ,\quad z\in\C,
$$ 
for $k\in\N$.
 
\begin{prop}\label{tempgrowthprop}
Let $u_{k,\alpha}$ be as in \eqref{ukafunction} 
for some $k\in\N$ and $\alpha>-1$. Then   
\begin{equation*}
\lvert u_{k,\alpha}(z)\rvert 
\leq \max_{0<\theta<\pi} 
\sin^{k+2\alpha+2}(\theta) \lvert p_{k,\alpha}(e^{i\theta})\rvert 
\Big(\lvert z\rvert^2/\im(z)\Big)^{k+\alpha+1}
\end{equation*}
for $z\in \Hu$. In particular, 
if $u\in\mathcal{V}_{\alpha,n}$ for some $n\in\N$, 
then $u$ has order of growth at most $n+\alpha+1$ at infinity.
\end{prop}

\begin{proof} 
Let $z\in\Hu$ and write $z=te^{i\theta}$ with $t>0$ and $0<\theta<\pi$.
By homogeneity 
we have that 
$$
u_{k,\alpha}(z)
=\sin^{\alpha+1}(\theta) p_{k,\alpha}(e^{i\theta})t^{k+\alpha+1}. 
$$
Notice also that $\lvert z\rvert^2/\im(z)=t/\sin(\theta)$. 
From these two facts we have that  
\begin{equation*}
u_{k,\alpha}(z)=
\sin^{k+2\alpha+2}(\theta) p_{k,\alpha}(e^{i\theta})  
\Big(\lvert z\rvert^2/\im(z)\Big)^{k+\alpha+1}
\end{equation*} 
for $z\in\Hu$, where $\theta=\arg z$. 
This yields the conclusion of the proposition. 
\end{proof}

The next result points out the significance of 
the class $\mathcal{V}_\alpha$.

\begin{cor}\label{Valphacharacterization}
Let $\alpha>-1$. Then a function $u$ belongs to 
the class $\mathcal{V}_\alpha$ 
if and only if it is $\alpha$-harmonic in $\Hu$, 
has temperate growth at infinity and 
vanishes on the real line in the sense that  
$u(z)\to0$ as $\Hu\ni z\to x$ for every $x\in\R$. 
\end{cor}

\begin{proof}
The if part is a re\-statement of Theorem \ref{alphaobstructionclass}.
From Corollary \ref{reprtermharmonic} we have 
that every function $u\in \mathcal{V}_\alpha$ 
is $\alpha$-harmonic in $\Hu$. 
From Proposition \ref{tempgrowthprop} we have that 
every function $u\in\mathcal{V}_\alpha$ is of temperate growth at infinity. 
It is evident that every function $u\in \mathcal{V}_\alpha$ 
vanishes on the real line. 
\end{proof}

Corollary \ref{Valphacharacterization} explains how 
the space $\mathcal{V}_\alpha$ 
can be thought of as the class of obstructions for 
the uniqueness problem for $\alpha$-harmonic functions in $\Hu$ 
with respect to a vanishing Dirichlet boundary value on the real line 
and temperate growth at infinity.

\begin{lemma}\label{ulimitlemma} 
Let $u\in\mathcal{V}_{\alpha,n}$ be of the form \eqref{alphaobstructionfcn} 
for some $\alpha>-1$ and $n\in\N$. 
Then 
\begin{equation*}
\lim_{t\to+\infty}u(te^{i\theta})/t^{n+\alpha+1}=c_n \sin^{\alpha+1}(\theta) p_{n,\alpha}(e^{i\theta})
\end{equation*}
for $0<\theta<\pi$. 
\end{lemma}

\begin{proof}
Let $t>0$ and $0<\theta<\pi$. 
By homogeneity 
we have that 
$$
u(te^{i\theta})=
\sum_{k=0}^n  c_k(\sin \theta)^{\alpha+1} p_{k,\alpha}(e^{i\theta}) t^{k+\alpha+1}.
$$
The result now follows by a passage to the limit.
\end{proof}

We shall dissect the space $\mathcal{V}_\alpha$ using 
orders of growth at infinity. 
We say that a function $u$ in $\Hu$ has order of growth $n+\alpha+1$ 
at infinity if it satisfies an estimate of the form 
\begin{equation}\label{utempgrowth}
\lvert u(z)\rvert\leq C(\lvert z\rvert^2/\im(z))^{n+\alpha+1}
\end{equation}
for $z\in\Hu$ with $\lvert z\rvert>R$, 
where $C\geq0$ and $R>1$ are finite constants.

\begin{lemma}\label{coefficientlemma}
Let $u\in \mathcal{V}_\alpha$ for some $\alpha>-1$. Let $n\in\N$. 
Then $u\in \mathcal{V}_{\alpha,n}$ if and only if 
\eqref{utempgrowth} holds.
Moreover, there exists a constant $C'=C'_{\alpha,n,R}$ such that 
$\lvert c_k\rvert\leq C'C$ for $k=0,1,\dots,n$ whenever 
$u\in\mathcal{V}_{\alpha,n}$ has the form \eqref{alphaobstructionfcn} 
and $C$ is as in \eqref{utempgrowth}.
\end{lemma}

\begin{proof}
From Proposition \ref{tempgrowthprop} we know that \eqref{utempgrowth} 
holds if $u\in\mathcal{V}_{\alpha,n}$. 
Assume next that $u\in\mathcal{V}_{\alpha}$ satisfies \eqref{utempgrowth}. 
We shall prove that $u\in\mathcal{V}_{\alpha,n}$.  
If $u\in\mathcal{V}_{\alpha,0}$ there is nothing to prove. 
Assume therefore that $u\in\mathcal{V}_{\alpha,m}$ 
and $u\not\in\mathcal{V}_{\alpha,m-1}$ for some $m\in\Z^+$.
An application of Lemma \ref{ulimitlemma} 
with $0<\theta<\pi$ chosen such that $p_{m,\alpha}(e^{i\theta})\neq0$ 
shows that $u$ has order of growth at least $m+\alpha+1$. 
Therefore $m\leq n$, so that $u\in\mathcal{V}_{\alpha,n}$.

We equip the space $\mathcal{V}_{\alpha,n}$ 
with the semi-norm given by the best constant $C$ in \eqref{utempgrowth}.
Another application of Lemma \ref{ulimitlemma} shows that this latter 
semi-norm is in fact a norm. The space $\mathcal{V}_{\alpha,n}$ is thus a 
normed complex vector space of finite dimension $n+1$. 
Since any two norms on such a space are equivalent, 
we conclude that  
$\lvert c_k\rvert\leq C'C$ for $k=0,1,\dots,n$ 
whenever $u\in \mathcal{V}_{\alpha,n}$.
\end{proof}

We denote by $\De'(\R)$ the space of distributions on 
the real line $\R$. An integrable function $f\in L^1(\R)$ 
is identified with the distribution  
$$
\langle f,\varphi\rangle=\int_{-\infty}^\infty f(x)\varphi(x)\, dx,
\quad \varphi \in C_0^\infty(\R), 
$$ 
where $C_0^\infty(\R)$ is the space of indefinitely differentiable 
test functions on $\R$ with compact support. 
By $u_j\to u$ in $\De'(\R)$ as $j\to\infty$ we understand that 
$\lim_{j\to\infty}\langle u_j,\varphi\rangle = \langle u,\varphi\rangle$ 
for every  $\varphi \in C_0^\infty(\R)$. 
A standard reference for distribution theory is H\"ormander \cite{Hormander}. 

We shall next extend the validity of Theorem \ref{alphaobstructionclass} 
to allow for a distributional boundary value.

\begin{theorem}\label{alphaobstructionclassdistributionalversion}
Let $\alpha>-1$.  
Let $u$ be an $\alpha$-harmonic function in $\Hu$ which is 
of temperate growth at infinity. 
Assume that $\lim_{y\to0+} u(\cdot+iy)=0$ in $\De'(\R)$.
Then $u\in \mathcal{V}_\alpha$. 
\end{theorem}

\begin{proof}
The theorem is proved by a regularization argument. 
Let $\psi\in C_0^\infty(\R)$ be a non\-negative compactly supported 
test function with $\int_{-\infty}^\infty\psi(x)\, dx=1$ and  
set $\psi_\varepsilon(x)=\psi(x/\varepsilon)/\varepsilon$ 
for $x\in\R$ and   $\varepsilon>0$.
We consider the regularizations 
$$
u_\varepsilon(z)=\int_{-\infty}^\infty u(z-t)\psi_\varepsilon(t)\, dt,\quad z\in\Hu,
$$ 
for $\varepsilon>0$. 
A differentiation under the integral shows that 
the function $u_\varepsilon$ is $\alpha$-harmonic in $\Hu$.  
It is straight\-forward to check that  
$u_\varepsilon(z)\to0$ as $\Hu\ni z\to x$ for every $x\in\R$. 
Using that $u$ is of temperate growth at infinity it 
is straight\-forward to check that  
\begin{equation}\label{tempgrowthbound}
\lvert u_\varepsilon(z)\rvert\leq C(\lvert z\rvert^2/\im(z))^{n+\alpha+1}
\end{equation}
for $z\in\Hu$ with $\lvert z\rvert>R$ and $0<\varepsilon<1$, 
where $C\geq0$ and $R>1$ are finite constants and $n\in\N$. 
Notice also that $u=\lim_{\varepsilon\to0}u_\varepsilon$ in $\Hu$ in the 
sense of normal convergence.

We now proceed to details.
By Theorem \ref{alphaobstructionclass} we have that 
$u_\varepsilon\in \mathcal{V}_\alpha$.  
From \eqref{tempgrowthbound} and Lemma \ref{coefficientlemma} 
we have that $u_\varepsilon\in \mathcal{V}_{\alpha,n}$.   
Therefore 
\begin{equation}\label{alphaobstructionfcnregularization}
u_\varepsilon(z)=\sum_{k=0}^{n}  c_k(\varepsilon)
(\im z)^{\alpha+1} p_{k,\alpha}(z),\quad z\in\Hu,
\end{equation}
for some $c_0(\varepsilon),\dots,c_{n}(\varepsilon)\in\C$.  
Another application of Lemma \ref{coefficientlemma} gives that   
$\lvert c_k(\varepsilon)\rvert\leq C'C$ for $k=0,1,\dots,n$, 
where $C$ is as in \eqref{tempgrowthbound} and $C'=C'_{\alpha,n,R}$ 
is a positive constant. 
By a compactness argument we can extract a subsequence 
$\varepsilon=\varepsilon_j\to 0$ of positive real numbers such that  
$c_k(\varepsilon_j)\to c_k$ as $j\to\infty$ for $k=0,1,\dots,n$. 
The conclusion of the theorem now follows by setting 
$\varepsilon=\varepsilon_j$ in \eqref{alphaobstructionfcnregularization} 
and letting $j\to\infty$.
\end{proof}

We point out that the 
assumption in 
Theorem \ref{alphaobstructionclassdistributionalversion} is that 
$$
\lim_{y\to 0}\int_{-\infty}^\infty u(x+iy)\varphi(x)\, dx=0
$$ 
for every $\varphi \in C_0^\infty(\R)$. 

The case of usual harmonic functions deserves special mention.

\begin{cor}\label{harmonicobstructionclassdistributionalversion}
Let $u$ be a harmonic function in $\Hu$ which is 
of temperate growth at infinity. 
Assume that $\lim_{y\to0+} u(\cdot+iy)=0$ in $\De'(\R)$.
Then $u\in \mathcal{V}_0$. 
\end{cor}

We record also the following 
slight improvement of the if part of Lemma \ref{coefficientlemma}.

\begin{prop}\label{Vanfromrelaxedgrowth}
Let $u\in \mathcal{V}_\alpha$ for some $\alpha>-1$. Let $n\in\N$. 
Assume that 
\begin{equation*}
u(z)=o((\lvert z\rvert^2/\im(z))^{n+\alpha+2})
\end{equation*}
as $\Hu\ni z\to\infty$ in the extended complex plane. 
Then $u\in \mathcal{V}_{\alpha,n}$. 
\end{prop}

\begin{proof} 
If $u\in\mathcal{V}_{\alpha,0}$ there is nothing to prove. 
Assume therefore that $u\in\mathcal{V}_{\alpha,m}$ 
and $u\not\in\mathcal{V}_{\alpha,m-1}$ for some $m\in\Z^+$.
An application of Lemma \ref{ulimitlemma} 
with $0<\theta<\pi$ chosen such that $p_{m,\alpha}(e^{i\theta})\neq0$ 
shows that $u$ has order of growth at least $m+\alpha+1$ at infinity. 
Therefore $m< n+1$, so that $u\in\mathcal{V}_{\alpha,n}$.
\end{proof} 

Proposition \ref{Vanfromrelaxedgrowth} is included for 
the sake of convenience.
We shall prove more refined versions of 
Proposition \ref{Vanfromrelaxedgrowth} in later sections.

\section{A uniqueness result for the case \texorpdfstring{$\alpha\neq0$}{nonzero alpha}}
\label{sectionuniguenessalphaneq0}

We now turn our attention to uniqueness results 
for $\alpha$-harmonic functions in $\Hu$. 
In this section we study the case of parameters $\alpha>-1$ 
such that $\alpha\neq 0$. 
The special case of usual harmonic functions is investigated 
in later sections. 
We begin our analysis with a study of 
the zeros of the polynomials $p_{k,\alpha}$.

A classical result of Enestr\"om-Kakeya says that if 
\begin{equation}\label{analyticpolynomial}
p(z)=\sum_{k=0}^n a_kz^k
\end{equation}
is an analytic polynomial of degree $n\geq0$ such that 
$a_k\leq a_{k+1}$ for $0\leq k<n$, then the zeroes of $p$ 
are all located in the closed unit disc $\bar\D$. 
An easy proof of this result can be found in 
Gardner and Govil \cite[Theorem 1.3]{gardner2014enestrom}.

\begin{cor}\label{EKcor} 
Let $p$ be an analytic polynomial of the form \eqref{analyticpolynomial} 
with positive coefficients: $a_k>0$ for $0\leq k\leq n$. 
Set 
$$
r=\min_{0\leq k<n}a_k/a_{k+1}\quad  \text{and}
\quad  R=\max_{0\leq k<n}a_k/a_{k+1}.
$$ 
Then the zeroes of $p$ are all located in the annulus  
$\{z\in\C:\ r\leq\lvert z\rvert\leq R\}$.
\end{cor}

\begin{proof}
We first show that the zeroes of $p$ are all located in the disc 
$\{z\in\C:\ \lvert z\rvert\leq R\}$.
Consider the polynomial $g(z)=p(Rz)$. By assumption the polynomial $g$ 
satisfies the assumptions of the Enestr\"om-Kakeya theorem quoted above. 
We conclude that the zeroes of $g$ are all located in $\bar\D$. 
This yields that the zeroes of $p$ are all located in the disc 
$\{z\in\C:\ \lvert z\rvert\leq R\}$. 

We next show that the zeroes of $p$ are all located in the exterior disc 
$\{z\in\C:\ \lvert z\rvert\geq r\}$. 
Consider the polynomial $h(z)=z^np(1/z)$. Notice that 
$$
h(z)=\sum_{k=0}^n a_{n-k}z^k.
$$
By the result of the previous paragraph, 
the zeroes of $h$ are all located in the disc 
$\{z\in\C:\ \lvert z\rvert\leq 1/r\}$.  
Therefore the zeroes of $p$ are all located in the exterior disc 
$\{z\in\C:\ \lvert z\rvert\geq r\}$.  
\end{proof}

Our interest in zeroes of polynomials stems from the following result.  

\begin{theorem}\label{pkakzerofree}
Let $\alpha>-1$ and $\alpha\neq0$. 
Let $p_{k,\alpha}$ be as in \eqref{eq:pkallealfa} for some $k\in\N$.
Then $p_{k,\alpha}(e^{i\theta})\neq0$ for $e^{i\theta}\in\T$.
\end{theorem}

\begin{proof}
By \eqref{eq:pkallealfa} it suffices to show that 
$s_{k,\alpha}(e^{i\theta})\neq0$ for $e^{i\theta}\in\T$, 
where $s_{k,\alpha}$ is as in \eqref{eq:skallealfa}.
Set $a_{j}(\alpha)=\frac{1}{j!}(\alpha+1)_j$. 
Observe that 
\begin{equation}\label{coefficientquotient}
\frac{a_{j}(\alpha)}{a_{j+1}(\alpha)}=\frac{j+1}{\alpha+j+1}
=1-\frac{\alpha}{\alpha+j+1}
\end{equation}
for $0\leq j<k$.

Assume next that $\alpha>0$. In this case the quotients in 
\eqref{coefficientquotient} increase in $j$ and we have that 
$a_{j}(\alpha)/a_{j+1}(\alpha)\leq k/(\alpha+k)$
for $0\leq j<k$. By Corollary \ref{EKcor} the 
zeroes of $s_{k,\alpha}$ are all located 
in the disc $\{z\in\C:\ \lvert z\rvert\leq k/(\alpha+k)\}$ which 
is compactly contained in $\D$.

Assume next that $-1<\alpha<0$. In this case the quotients in 
\eqref{coefficientquotient} decrease in $j$ and we have that 
$a_{j}(\alpha)/a_{j+1}(\alpha)\geq 1/(\alpha+1)$
for $0\leq j<k$. By Corollary \ref{EKcor} the 
zeroes of $s_{k,\alpha}$ are all located 
in the exterior disc $\{z\in\C:\ \lvert z\rvert\geq 1/(\alpha+1)\}$ which 
is disjoint from $\bar\D$.
\end{proof}

\begin{remark}
We point out that the zero set of $p_{k,0}$ intersects the unit circle 
if $k\geq1$. In fact, from \eqref{eq:skallealfa} we have that 
$$
s_{k,0}(z)=\sum_{j=0}^kz^j=(1-z^{k+1})/(1-z)
$$   
for $k\in\N$.
\end{remark}

We now consider the class $\mathcal{V}_{\alpha}$. 

\begin{theorem}\label{Vauniquenessaneq0}
Let $u\in\mathcal{V}_{\alpha}$ for some $\alpha>-1$ with $\alpha\ne0$.
Assume that there is a sequence $\{z_j\}$ in $\Hu$ 
with $\lvert z_j\rvert \to\infty$ as $j\to\infty$ such that 
\begin{equation}\label{sequencevanishing}
\lim_{j\to\infty} \frac{u(z_j)}{(\im(z_j))^{\alpha+1}}=0.
\end{equation} 
Then $u(z)=0$ for all $z\in\Hu$.
\end{theorem}

\begin{proof}
We put $u\in\mathcal{V}_{\alpha}$ on the form \eqref{alphaobstructionfcn} 
for some $n\in\N$ and constants $c_0,\dots,c_n\in\C$, 
where the $p_{k,\alpha}$'s are as in \eqref{eq:pkallealfa}.
Recall that the polynomial $p_{k,\alpha}$ is homogeneous of degree $k\in\N$. 
By Theorem \ref{pkakzerofree}  the polynomial $p_{k,\alpha}$ 
has no zeroes on the unit circle. 
From these two properties of $p_{k,\alpha}$ we have that the quotient 
$\lvert p_{k,\alpha}(z)\rvert/\lvert z\rvert^k$ 
is bounded from above and below by finite positive constants  
uniformly in the punctured plane.  
From \eqref{sequencevanishing}  we now have that 
$$
0=\lim_{j\to\infty} \frac{u(z_j)}{p_{n,\alpha}(z_j)(\im(z_j))^{\alpha+1}}=c_n.
$$
Repeating the argument we conclude that $c_k=0$ for $0\leq k\leq n$.  
Therefore $u(z)=0$ for all $z\in\Hu$.
\end{proof}

We emphasize that condition \eqref{sequencevanishing} can be checked 
along any sequence $\{z_j\}$ in $\Hu$ such that 
$z_j\to\infty$ in the extended complex plane $\C_\infty$ as $j\to\infty$.

We now turn to uniqueness theorems for $\alpha$-harmonic functions 
with $\alpha\neq0$. 
We first prove Theorem \ref{introuniquenessthm} stated in the introduction. 

\begin{proof}[Proof of Theorem \ref{introuniquenessthm}]
By assumption \eqref{classicalvanishingrealline} 
(formula \eqref{vanishingDirichletbdrycondition}) 
we can apply  
Theorem \ref{alphaobstructionclass} to conclude that 
$u\in\mathcal{V}_\alpha$.
We next apply Theorem \ref{Vauniquenessaneq0}
to conclude that $u(z)=0$ for $z\in\Hu$.
\end{proof}

We next extend the validity of Theorem \ref{introuniquenessthm} 
to allow for a distributional boundary value instead 
of \eqref{vanishingDirichletbdrycondition}.

\begin{theorem}\label{thm:distributionaluniquenessaneq0}
Let $\alpha>-1$ and $\alpha\ne0$.
Let $u$ be an $\alpha$-harmonic function in $\Hu$ which is 
of temperate growth at infinity. 
\begin{enumerate}
\item[(1)']
Assume that 
$\lim_{y\to0+} u(\cdot+iy)=0$ in $\De'(\R)$.
\item[(2)] 
Assume that there is a sequence $\{z_j\}$ in $\Hu$ 
with $z_j\to\infty$ in $\C_\infty$ as $j\to\infty$ such that 
\eqref{sequencevanishing} holds.
\end{enumerate} 
Then $u(z)=0$ for all $z\in\Hu$.
\end{theorem}

\begin{proof}
In view of assumption \eqref{classicalvanishingrealline}' we can apply  
Theorem \ref{alphaobstructionclassdistributionalversion} to conclude that 
$u\in\mathcal{V}_\alpha$.   
We next apply Theorem \ref{Vauniquenessaneq0}
to conclude that $u(z)=0$ for $z\in\Hu$.
\end{proof}

We close this section with a relaxed version of 
Theorem \ref{Vauniquenessaneq0}.

\begin{theorem}\label{Vancharacterizationaneq0}
Let $u\in\mathcal{V}_{\alpha}$ for some $\alpha>-1$ with $\alpha\neq 0$.
Let $n\in\N$. Assume that there is a sequence $\{z_j\}$ in $\Hu$ 
with $z_j\to\infty$ in $\C_\infty$ as $j\to\infty$ such that 
\begin{equation*}
\lim_{j\to\infty} \frac{u(z_j)}{\lvert z_j\rvert^{n+1}(\im(z_j))^{\alpha+1}}=0.
\end{equation*} 
Then $u\in\mathcal{V}_{\alpha,n}$.
\end{theorem}
 
The proof of Theorem \ref{Vancharacterizationaneq0} follows along 
the same lines as the proof of Theorem \ref{Vauniquenessaneq0}. 
We omit the details.

\section{Harmonic functions vanishing along geodesics}

We shall now turn to uniqueness results for classical harmonic 
functions in $\Hu$ ($\alpha=0$). 
Recall the harmonic polynomials 
$$
u_{k,0}(z)=\im(z^{k+1}),\quad z\in\C,
$$
for $k\in\N$ which all vanish on the real line $\R$. 
Observe that the zero set of $u_{k,0}$ is a union of $k+1$ 
lines passing through the origin. 
A treatment of the uniqueness problem for classical harmonic 
functions along the lines of what we did in 
Section \ref{sectionuniguenessalphaneq0} 
thus calls for a more demanding vanishing condition at infinity.  
In this section we shall consider vanishing conditions 
along geodesics in $\Hu$, that is, along 
rays in $\Hu$ that are parallel to the positive imaginary axis. 

Let us first consider the class $\mathcal{V}_{0}$.  

\begin{theorem}\label{V0uniquenessgeodesic}
Let $u\in\mathcal{V}_{0}$.  
Assume that 
\begin{equation}\label{geodesicvanishing}
\lim_{y\to+\infty}u(x+iy)/y=0
\end{equation}
for $x=x_j\in\R$ ($j=1,2$) with $x_1\neq x_2$. 
Then $u(z)=0$ for all $z\in\Hu$.
\end{theorem}

\begin{proof}  
We put $u\in\mathcal{V}_{0}$ on the form \eqref{harmonicobstructionfcn} 
for some $c_1,\dots, c_{n+1}\in\C$ and $n\in\N$. 
By adding an extra term in \eqref{harmonicobstructionfcn} 
if necessary, we can arrange that $n\geq1$ is odd. 
We shall prove that $c_n=c_{n+1}=0$. 
This will complete the proof of the theorem.

We now proceed to details. Let $x\in\R$. 
From the binomial theorem we have that 
\begin{align*}
(x+iy)^{n+1}&= (iy)^{n+1}+(n+1)x(iy)^{n} +O(y^{n-1})\\
&=(-1)^{(n+1)/2}y^{n+1}+i(n+1)(-1)^{(n-1)/2}xy^{n}+O(y^{n-1})
\end{align*} 
as $y\to +\infty$. Passing to the imaginary part we have that 
\begin{equation}\label{(n+1)thtermasymptotics}
\im((x+iy)^{n+1})=(n+1)(-1)^{(n-1)/2}xy^{n}+O(y^{n-1})
\end{equation}
as $y\to +\infty$. A similar consideration shows that 
\begin{equation}\label{nthtermasymptotics}
\im((x+iy)^{n})=(-1)^{(n-1)/2}y^{n}+O(y^{n-1})
\end{equation}
as $y\to +\infty$. 

We now return to the function $u$. From \eqref{(n+1)thtermasymptotics} 
and \eqref{nthtermasymptotics} we have that 
$$
\lim_{y\to+\infty}u(x+iy)/y^{n}=c_{n+1} (n+1)(-1)^{(n-1)/2}x+c_{n}(-1)^{(n-1)/2}
$$ 
for $x\in\R$. From \eqref{geodesicvanishing} we have that   
$$
c_{n+1} (n+1)(-1)^{(n-1)/2}x+c_{n}(-1)^{(n-1)/2}=0
$$
for $x=x_j$ ($j=1,2$). Since $x_1\neq x_2$, we conclude 
that $c_{n+1}=c_{n}=0$.  
\end{proof}

We next turn to uniqueness theorems for harmonic functions.

\begin{theorem}\label{classicaluniquenessthm2geodesics}  
Let $u$ be a harmonic function in the open upper half-plane $\Hu$ 
which is of temperate growth at infinity. 
\begin{enumerate}
\item \label{classicalvanishingonrealboundary}
Assume that $u(z)\to 0$ as $\Hu\ni z\to x$ for every $x\in\R$.
\item\label{vanishingalonggeodesics}
Assume that 
$$
\lim_{y\to+\infty}u(x+iy)/y=0
$$ 
for $x=x_j\in\R$ ($j=1,2$) with $x_1\neq x_2$. 
\end{enumerate} 
Then $u(z)=0$ for all $z\in\Hu$.
\end{theorem}

\begin{proof}  
From Theorem \ref{alphaobstructionclass} we have that $u\in\mathcal{V}_0$. 
An application of Theorem \ref{V0uniquenessgeodesic} yields that 
$u(z)=0$ for $z\in\Hu$.
\end{proof}

Recall the definition of $u_{k,\alpha}$ in \eqref{ukafunction} and let $a\in\R$. We point out that the function 
$$
u(z)=u_{1,0}(z-a)=\im((z-a)^2)=2(x-a)y,\quad z=x+iy\in\C,
$$
is of quadratic growth, harmonic, and vanishes on the lines $x=a$ and $y=0$. 
Thus, the vanishing assumption \eqref{vanishingalonggeodesics}  
in Theorem \ref{classicaluniquenessthm2geodesics}  
can not be relaxed to vanishing along one geodesic $x=a$ only.

We next extend the validity of 
Theorem \ref{classicaluniquenessthm2geodesics}  
to allow for a distributional boundary value in 
\eqref{classicalvanishingonrealboundary}.

\begin{theorem}\label{distributionaluniquenessthm2geodesics}  
Let $u$ be a harmonic function in the open upper half-plane $\Hu$ 
which is of temperate growth at infinity. 
\begin{enumerate}
\item[(1)'] 
Assume that $\lim_{y\to0+} u(\cdot+iy)=0$ in $\De'(\R)$.
\item[(2)]
Assume that 
$$
\lim_{y\to+\infty}u(x+iy)/y=0
$$ 
for $x=x_j\in\R$ ($j=1,2$) with $x_1\neq x_2$. 
\end{enumerate} 
Then $u(z)=0$ for all $z\in\Hu$.
\end{theorem}

\begin{proof}  
From Theorem \ref{alphaobstructionclassdistributionalversion}  
we have that $u\in\mathcal{V}_0$. 
An application of Theorem \ref{V0uniquenessgeodesic} yields that 
$u(z)=0$ for $z\in\Hu$.
\end{proof}

The results of this section together with corresponding results 
from Section \ref{sectionuniguenessalphaneq0} yield 
considerable advancement from a recent 
result by Carlsson and Wittsten \cite[Corollary 1.9
]{carlsson2016dirichlet}. 
In fact, instead of \eqref{sequencevanishing} 
or \eqref{geodesicvanishing} 
those authors considered functions such that 
$$
\lim_{y\to\infty}u(x+iy)/y^{\alpha+1}=0
$$   
for all $\lvert x\rvert<\delta$, where $\delta>0$.

The importance of condition \eqref{geodesicvanishing} 
for two distinct geodesics brings to mind an analogous situation 
in the study non-tangential limits. 
If $u$ is a bounded harmonic function in $\Hu$ such that 
$$
\lim_{r\to0}u(re^{i\theta_1})=L=\lim_{r\to0}u(re^{i\theta_2})
$$ 
for some $0<\theta_1<\theta_2<\pi$, 
then $u$ has non-tangential limit $L$ at the origin 
(see Axler et al. \cite[Theorem 2.10]{ABR}).
This latter result can be thought of as a harmonic function version 
of a classical result of Lindel\"of (see Rudin \cite[Theorem 12.10]{Rudin}).

\section{Harmonic functions vanishing along rays}\label{section:rays}

In this section we continue our study of 
uniqueness results for classical harmonic 
functions in $\Hu$ ($\alpha=0$). 
We now consider functions vanishing along rays in $\Hu$ 
emanating from the origin, that is, 
along rays of the form $\{te^{i\theta}:\ t>0\}$, where $0<\theta<\pi$.

We shall use a concept of admissible function of angles 
that we proceed to define.

\begin{dfn}\label{dfnadmissiblefcnangles}
Let $E\subset(0,\pi)$ be a set and $\eta:E\to\Z^+$ 
a positive integer valued function on $E$. 
We say that the function element $(E,\eta)$ 
is an {\it admissible function of angles} if it has the property that 
for every $k\in\Z^+$ there exists $\theta\in E$ 
such that $\sin(k\theta)\neq0$ and $k\geq\eta(\theta)$.
\end{dfn}

Our interest in admissible functions of angles stems from the following theorem.

\begin{theorem}\label{V0uniquenessray}
Let $u\in\mathcal{V}_{0}$.    
Assume that there is an admissible function of angles $(E,\eta)$ such that 
$$
\lim_{t\to+\infty}u(te^{i\theta})/t^{\eta(\theta)}=0
$$ 
for every $\theta\in E$. 
Then $u(z)=0$ for all $z\in\Hu$.
\end{theorem}
 
\begin{proof}
We put $u\in\mathcal{V}_{0}$ on the form 
$$
u(z)=\sum_{k=1}^n c_k\im(z^k),\quad z\in\Hu,
$$
for some $c_1,\dots,c_n\in\C$ and $n\in \Z^+$. 
It suffices to show that $c_n=0$. For such functions, 
$$
\lim_{t\to+\infty}u(e^{i\theta}t)/t^{n}
=\lim_{t\to+\infty}\sum_{k=1}^nc_k\im(e^{ik\theta}t^k)/t^n=c_n \sin(n\theta).
$$
Since $(E,\eta)$ is an admissible function of angles we may choose 
$\theta\in E $ such that $\sin(n\theta)\ne0$ and $\eta(\theta)\le n$. 
It follows that
$$
c_n\sin(n\theta)=\lim_{t\to+\infty}\frac{u(te^{i\theta})}{t^{\eta(\theta)}}\frac{t^{\eta(\theta)}}{t^{n}}=0.
$$
Since $\sin(n\theta)\ne0$ we conclude that $c_n=0$.
\end{proof}

We next turn to uniqueness results. 

\begin{theorem}\label{classicalfcnofanglesvanishing}  
Let $u$ be a harmonic function in the open upper half-plane $\Hu$ 
which is of temperate growth at infinity. 
\begin{enumerate}
\item 
Assume that $u(z)\to 0$ as $\Hu\ni z\to x$ for every $x\in\R$.
\item\label{vanishingatfcnofangles}
Assume that there is an admissible function of angles $(E,\eta)$ such that 
$$
\lim_{t\to+\infty}u(te^{i\theta})/t^{\eta(\theta)}=0
$$ 
for every $\theta\in E$. 
\end{enumerate} 
Then $u(z)=0$ for all $z\in\Hu$.
\end{theorem}

\begin{proof}  
From Theorem \ref{alphaobstructionclass} we have that $u\in\mathcal{V}_0$. 
An application of Theorem \ref{V0uniquenessray} yields that 
$u(z)=0$ for $z\in\Hu$.
\end{proof}

We next extend the validity of 
Theorem \ref{classicalfcnofanglesvanishing}    
to allow for a distributional boundary value in 
\eqref{classicalvanishingonrealboundary}.

\begin{theorem}\label{fcnofanglesvanishing}  
Let $u$ be a harmonic function in the open upper half-plane $\Hu$ 
which is of temperate growth at infinity. 
\begin{enumerate}
\item[(1)'] 
Assume that $\lim_{y\to0+} u(\cdot+iy)=0$ in $\De'(\R)$.
\item[(2)]
Assume that there is an admissible function of angles $(E,\eta)$ such that 
$$
\lim_{t\to+\infty}u(te^{i\theta})/t^{\eta(\theta)}=0
$$ 
for every $\theta\in E$. 
\end{enumerate} 
Then $u(z)=0$ for all $z\in\Hu$.
\end{theorem}

\begin{proof}  
From Theorem \ref{alphaobstructionclassdistributionalversion}  
we have that $u\in\mathcal{V}_0$. 
An application of Theorem \ref{V0uniquenessray} yields that 
$u(z)=0$ for $z\in\Hu$.
\end{proof}

Let us comment on the generic situation.

\begin{cor}\label{genericrayvanishing}
Let $u$ be a harmonic function in the open upper half-plane $\Hu$ 
which is of temperate growth at infinity.
\begin{enumerate}
\item[(1)'] 
Assume that $\lim_{y\to0+} u(\cdot+iy)=0$ in $\De'(\R)$.
\item[(2)]
Assume that $u$ satisfies the vanishing condition that 
$\lim_{t\to+\infty}u(te^{i\theta})/t=0$ for some $0<\theta<\pi$ 
which is not a rational multiple of $\pi$.
\end{enumerate} 
Then $u(z)=0$ for all $z\in\Hu$.
\end{cor}

\begin{proof}
Let $E=\{\theta\}$ and $\eta(\theta)=1$. Then $(E,\eta)$ is an admissible function of angles for which condition (2) of Theorem \ref{fcnofanglesvanishing} holds. Thus, an application of the mentioned theorem gives the result.
\end{proof}

Theorem \ref{classicalgenericrayuniquenessresult} in 
the introduction is proved in the same way as 
Corollary \ref{genericrayvanishing} by using 
Theorem \ref{classicalfcnofanglesvanishing}  
instead of Theorem \ref{fcnofanglesvanishing}.

The examples  
$$
u_k(z)=\im(z^k)
$$ 
for $k=1,2,\dots$, show that the arithmetic condition on 
$\theta\in(0,\pi)$ in Corollary \ref{genericrayvanishing} 
is to the point. Clearly, the polynomial $u_k$ vanishes on $\R$. 
Moreover,  
if $\theta=\frac{m}{n}\pi$ for some integers $m$ and $n$ with $n\geq1$, 
then the function $u_n$ vanishes on the line $e^{i\theta}\R$.

\begin{cor}\label{rationalrayvanishing}
Let $u$ be a harmonic function in the open upper half-plane $\Hu$ 
which is of temperate growth at infinity.  
Let $\tau=\frac{m}{n}\pi$ with $m,n\in\Z^+$ relatively prime and $m<n$.
\begin{enumerate}
\item[(1)'] 
Assume that $\lim_{y\to0+} u(\cdot+iy)=0$ in $\De'(\R)$.
\item[(2)]
Assume that 
$\lim_{t\to+\infty}u(te^{i\tau})/t=0$.
\item[(3)] 
Assume that 
$\lim_{t\to+\infty}u(te^{i\theta})/t^n=0$ 
for some $0<\theta<\pi$ which is not a rational multiple of $\pi$.
\end{enumerate} 
Then $u(z)=0$ for all $z\in\Hu$.
\end{cor}

\begin{proof}
Let $E=\{\tau,\theta\}$ and set $\eta(\tau)=1$ and $\eta(\theta)=n$. Then $(E,\eta)$ is an admissible function of angles such that condition (2) of Theorem \ref{fcnofanglesvanishing} is satisfied. Applying the mentioned theorem thus gives the result.
\end{proof}

It might be worthwhile to state the special case of Corollary \ref{rationalrayvanishing} obtained when $\tau=\pi/2$:

\begin{cor}\label{testthm}
Let $u$ be a harmonic function in the open upper half-plane $\Hu$ 
which is of temperate growth at infinity.
\begin{enumerate}
\item[(1)']
Assume that $\lim_{y\to0+} u(\cdot+iy)=0$ in $\De'(\R)$.
\item[(2)]
Assume that $u$ satisfies the vanishing condition that 
$\lim_{y\to+\infty}u(iy)/y=0$. 
\item[(3)]
Assume that $u$ satisfies the vanishing condition that 
$\lim_{t\to+\infty}u(te^{i\theta})/t^2=0$ for some $0<\theta<\pi$ 
which is not a rational multiple of $\pi$.
\end{enumerate} 
Then $u(z)=0$ for all $z\in\Hu$.
\end{cor}

Let us elaborate some more on the theme of Corollary \ref{testthm}.

\begin{cor}\label{testthm2}
Let $u$ be a harmonic function in the open upper half-plane $\Hu$ 
which is of temperate growth at infinity.
\begin{enumerate}
\item[(1)'] 
Assume that $\lim_{y\to0+} u(\cdot+iy)=0$ in $\De'(\R)$.
\item[(2)]
Assume that 
$$
\lim_{t\to+\infty}u(te^{i2^{-k}\pi})/t^{2^{k-1}}=0
$$
for $k=1,2,\dots$.
\end{enumerate} 
Then $u(z)=0$ for all $z\in\Hu$.
\end{cor}

\begin{proof}
Let $E=\{2^{-k}\pi:\ k=1,2,\dots\}$ and 
set $\eta(2^{-k}\pi)=2^{k-1}$ for $k=1,2,\dots$.
From Theorem \ref{constructionfoainfinite} in the next section 
we have that $(E,\eta)$ is an admissible function of angles. 
Clearly condition (2) of Theorem \ref{fcnofanglesvanishing} is satisfied. 
Applying the mentioned theorem thus gives the result.
\end{proof}

\section{Admissible functions of angles}\label{section:construction}

In this section we shall provide some constructions of admissible functions 
of angles (see Theorems \ref{constructionfoainfinite} 
and \ref{constructionfoafinite}). 
Examples of such function elements are needed for successful applications of 
Theorems \ref{classicalfcnofanglesvanishing} or \ref{fcnofanglesvanishing}. 
The set of admissible functions of angles is equipped with a natural 
partial order.
We show that the admissible functions of angles constructed are in fact 
the minimal elements in the partial order 
(see Theorem \ref{minimalelementsA}). 
We also show that every admissible function of angles 
has a lower bound which is minimal (see Theorem \ref{lowerboundfoa}).  
The construction uses some notions from ideal theory 
for the ring of integers $\Z$.   

Let $\theta\in\R$ be a real number and consider the set of integers  
$$
\mathcal{I}(\theta)=\{m\in \Z: \ \sin(m\theta)=0  \}.
$$
Using standard properties of 
the sine function, 
it is straight\-forward to check that the set $\mathcal{I}(\theta)$ is an ideal in $\Z$.
Since the ring $\Z$ is a principal ideal domain, we have that 
$$
\mathcal{I}(\theta)=d(\theta)\Z
$$ 
for some integer $d(\theta)$. Since the units in $\Z$ are $\pm1$, 
this integer $d(\theta)$ is uniquely determined by the condition that $d(\theta)\geq0$.
This defines a function $d:\theta\mapsto d(\theta)$ from $\R$ to $\N$.

We observe that $d(\theta)=0$ if and only if 
$\theta\in\R$ is not a rational multiple of $\pi$. Furthermore,  
if $\theta=\frac{r}{s}\pi $ with $r\in\Z$ and $s\in \Z^+$ relatively prime, 
then $d(\theta)=s$. These assertions are straight\-forward to check.

We denote by $\operatorname{lcm}(a_1,\dots,a_n)$ the least common multiple of 
integers $a_1,\dots,a_n$. Recall that $\operatorname{lcm}(a_1,\dots,a_n)$ is 
the non-negative generator of the ideal $\cap_{k=1}^na_k\Z$:   
$$
\operatorname{lcm}(a_1,\dots,a_n)\Z=\cap_{k=1}^na_k\Z.
$$ 
The symbol $a\vert b$ means that the integer $b$ is divisible by the integer $a$  
in the usual sense that $b=ac$ for some integer $c$.

We are now ready for the construction of admissible functions of angles.  

\begin{theorem}\label{constructionfoainfinite}
Let $\{\theta_k\}_{k=1}^\infty$ be an infinite sequence of rational multiples of $\pi$ 
from the interval $(0,\pi)$ such that 
$$
d(\theta_k)\nmid \operatorname{lcm}(d(\theta_1),\dots,d(\theta_{k-1}))
$$
for $k=2,3,\dots$. Set $E=\{\theta_k:\ k=1,2,\dots \}$ and 
define $\eta:E\to\Z^+$ by $\eta(\theta_1)=1$ and 
$$
\eta(\theta_k)= \operatorname{lcm}(d(\theta_1),\dots,d(\theta_{k-1}))
$$
for $k=2,3,\dots$. Then $(E,\eta)$ is an admissible function of angles. 
\end{theorem}

\begin{proof}
Observe first that the requirements on the sequence $\{\theta_k\}_{k=1}^\infty$ 
guarantee that  
$$
2\leq \operatorname{lcm}(d(\theta_1),\dots,d(\theta_{j}))
<\operatorname{lcm}(d(\theta_1),\dots,d(\theta_{k}))
$$
for $1\leq j<k$. 
Thus  $\theta_j\neq\theta_k$ if $j\neq k$ and $\cap_{k=1}^\infty d(\theta_k)\Z=\{0\}$.

The function $\eta:E\to\Z^+$ is well-defined since the $\theta_k$'s are distinct. 
We proceed to prove that $(E,\eta)$ is an admissible function of angles. 
Let $m\in\Z^+$ be a positive integer. Since $\cap_{k=1}^\infty d(\theta_k)\Z=\{0\}$, 
there exists $k\in\Z^+$ such that $m\not\in d(\theta_k)\Z$. Furthermore, 
by the well-ordering of positive integers we can also arrange that 
$m\in d(\theta_j)\Z$ for $1\leq j<k$. Since  $m\not\in d(\theta_k)\Z$, 
we have that $\sin(m\theta_k)\neq0$. 
If $k=1$, we clearly have that $m\geq1=\eta(\theta_1)$. 
Assume next that $k\geq2$.  Since  $m\in d(\theta_j)\Z$ for $1\leq j<k$, 
we have 
$$
m\in \cap_{j=1}^{k-1}d(\theta_j)\Z
= \operatorname{lcm}(d(\theta_1),\dots,d(\theta_{k-1}))\Z,  
$$
which allows us to conclude that 
$m\geq \operatorname{lcm}(d(\theta_1),\dots,d(\theta_{k-1}))=\eta(\theta_k)$. 
This completes the proof of the theorem.
\end{proof}

Theorem \ref{constructionfoainfinite} provides a multitude of examples of 
admissible functions of angles.
For example, 
let $\theta_k=\frac{n_k}{2^k}\pi$ with $1\leq n_k<2^k$ odd for $k=1,2,\dots$. 
Now $d(\theta_k)=2^{k}$ for $k=1,2,\dots$, 
which makes evident that the assumption of Theorem \ref{constructionfoainfinite} 
is satisfied. In this case $\eta(\theta_k)=2^{k-1}$ for $k=1,2,\dots$.

We next turn to the construction of admissible functions of angles $(E,\eta)$ 
with $E$ finite.

\begin{theorem}\label{constructionfoafinite}
Let $\{\theta_k\}_{k=1}^n$ be a finite sequence of rational multiples of $\pi$ 
from the interval $(0,\pi)$ such that 
$$
d(\theta_k)\nmid \operatorname{lcm}(d(\theta_1),\dots,d(\theta_{k-1}))
$$
for $2\leq k\leq n$, where  $n\in\N$. 
Let $\theta_{n+1}\in(0,\pi)$ be an irrational multiple of $\pi$.
Set $E=\{\theta_k:\ 1\leq k\leq n+1 \}$ and 
define $\eta:E\to\Z^+$ by $\eta(\theta_1)=1$ and 
$$
\eta(\theta_k)= \operatorname{lcm}(d(\theta_1),\dots,d(\theta_{k-1}))
$$
for $2\leq k\leq n+1$. Then $(E,\eta)$ is an admissible function of angles. 
\end{theorem}

\begin{proof}
The function $\eta:E\to\Z^+$ is well-defined since the $\theta_k$'s are distinct. 
We proceed to prove that $(E,\eta)$ is an admissible function of angles. 
Let $m\in\Z^+$ be a positive integer. 
If $m\geq \operatorname{lcm}(d(\theta_1),\dots,d(\theta_{n}))$, then 
it is straight\-forward to check that 
$\sin(m\theta_{n+1})\neq0$ and  $m\geq\eta(\theta_{n+1})$.

It remains to consider the case  
$1\leq m< \operatorname{lcm}(d(\theta_1),\dots,d(\theta_{n}))$.    
Since 
\begin{equation*}
\cap_{k=1}^n d(\theta_k)\Z=\operatorname{lcm}(d(\theta_1),\dots,d(\theta_{n}))\Z,
\end{equation*}
there exists $k\in\Z^+$ with $1\leq k\leq n$ such that 
$m\not\in d(\theta_k)\Z$ and 
$m\in d(\theta_j)\Z$ for $1\leq j<k$. 
Since  $m\not\in d(\theta_k)\Z$, we have that $\sin(m\theta_k)\neq0$. 
If $k=1$, we clearly have that $m\geq1=\eta(\theta_1)$.
Assume next that $k\geq2$.  Since  $m\in d(\theta_j)\Z$ for $1\leq j<k$, 
we have 
$$
m\in \cap_{j=1}^{k-1}d(\theta_j)\Z
= \operatorname{lcm}(d(\theta_1),\dots,d(\theta_{k-1}))\Z,  
$$
which allows us to conclude that 
$m\geq \operatorname{lcm}(d(\theta_1),\dots,d(\theta_{k-1}))=\eta(\theta_k)$. 
This yields the conclusion of the theorem.
\end{proof}

We point out that the case of an empty sequence $\{\theta_k\}_{k=1}^n$ ($n=0$)
in Theorem \ref{constructionfoafinite} yields the admissible 
function of angles used in the proof of Corollary \ref{genericrayvanishing}.

Let $E_j\subset(0,\pi)$ and $\eta_j:E_j\to\Z^+$ ($j=1,2$) 
be such that $E_1\subset E_2$ and 
$\eta_1(\theta)\geq \eta_2(\theta)$ for $\theta\in E_1$. 
If $(E_1,\eta_1)$ is an admissible function of angles, 
then so is $(E_2,\eta_2)$. 
For applications of Theorems \ref{classicalfcnofanglesvanishing} 
or \ref{fcnofanglesvanishing}  
it is of interest to have an admissible 
function of angles $(E,\eta)$ which is as economical as possible 
in this respect.

Let us denote by $\mathcal{A}$ the set of all admissible functions of angles. 
We equip the set $\mathcal{A}$  with a relation ($\leq$) defined by 
$(E_1,\eta_1)\leq (E_2,\eta_2)$ 
if $E_1\subset E_2$ and $\eta_1(\theta)\geq\eta_2(\theta)$ for $\theta\in E_1$.
It is straight\-forward to check that this relation gives the set $\mathcal{A}$
a structure of a partial order.

We denote by $\mathcal{A}_0$ the sub\-set of $\mathcal{A}$  
consisting of those admissible functions of angles $(E,\eta)$ that are  
constructed as in Theorem \ref{constructionfoainfinite} 
or \ref{constructionfoafinite}.

We next prove that every element in $\mathcal{A}$ has a lower bound 
from the set $\mathcal{A}_0$.

\begin{theorem}\label{lowerboundfoa}
Let $(E_1,\eta_1)\in\mathcal{A}$ be arbitrary. 
Then there exists $(E,\eta)\in\mathcal{A}_0$ 
such that 
$(E,\eta)\leq (E_1,\eta_1)$ in $\mathcal{A}$. 
\end{theorem}

\begin{proof}
The proof relies on a construction of 
a suitable sequence $\{\theta_k\}$ from the set $E_1\subset (0,\pi)$. 
Since $(E_1,\eta_1)$ is an admissible function of angles 
there exists $\theta\in E_1$ such that $\eta_1(\theta)=1$. 
We set $\theta_1=\theta$. If $\theta_1$ is an ir\-rational multiple of $\pi$, 
then $(E,\eta)\leq (E_1,\eta_1)$ in $\mathcal{A}$, 
where $E=\{\theta_1\}$ and $\eta(\theta_1)=1$ 
are as in Theorem \ref{constructionfoafinite} with $n=0$ there. 
If $\theta_1$ is a rational multiple of $\pi$, 
then we proceed to construct an element $\theta_2$ as described in 
the following paragraph.

Assume that $\theta_1,\dots,\theta_{n}\in E_1$ are rational multiples of $\pi$ 
such that  
$$
d(\theta_k)\nmid  \operatorname{lcm}(d(\theta_1),\dots,d(\theta_{k-1}))
$$ 
and 
$$
\eta_1(\theta_k)\leq \operatorname{lcm}(d(\theta_1),\dots,d(\theta_{k-1}))
$$
for $2\leq k\leq n$, where $n\geq1$. 
Let $m=\operatorname{lcm}(d(\theta_1),\dots,d(\theta_{n}))$. Clearly $m\in\Z^+$. 
Since  $(E_1,\eta_1)$ is an admissible function of angles 
there exists $\theta\in E_1$ such that $\sin(m\theta)\neq 0$
and $m\geq \eta_1(\theta)$.
We set $\theta_{n+1}=\theta$. 
If $\theta_{n+1}$ is an irrational multiple of $\pi$, 
then the construction from Theorem \ref{constructionfoafinite} 
provides us with a function of angles $(E,\eta)\in\mathcal{A}_0$ such that 
$(E,\eta)\leq (E_1,\eta_1)$ in $\mathcal{A}$. 

Assume next that $\theta_{n+1}$ is a rational multiple of $\pi$. 
Since $\sin(m\theta_{n+1})\neq0$, we have 
$$
m\not\in \mathcal{I}(\theta_{n+1})=d(\theta_{n+1})\Z,
$$ 
which gives that 
$d(\theta_{n+1})\nmid m=\operatorname{lcm}(d(\theta_1),\dots,d(\theta_{n}))$. 
Notice also that the in\-equality 
$\eta_1(\theta_{n+1})\leq\operatorname{lcm}(d(\theta_1),\dots,d(\theta_{n}))$ 
is evident from construction. We are now in position to 
repeat the procedure from the previous paragraph.

Unless the above procedure terminates in a finite number of steps, 
the principle of induction provides us with with an 
in\-finite sequence $\{\theta_k\}_{k=1}^\infty$ of rational multiples 
of $\pi$ from the set $E_1$ satisfying the assumptions of 
Theorem \ref{constructionfoainfinite}. The construction from   
Theorem \ref{constructionfoainfinite} now provides us 
with an admissible function of angles $(E,\eta)\in\mathcal{A}_0$ 
such that $(E,\eta)\leq (E_1,\eta_1)$ in $\mathcal{A}$.
\end{proof}

Recall that an element $m$ in a partial order $\mathcal{M}$ is called 
minimal (in $\mathcal{M}$)  if 
it has the property that $x\in \mathcal{M}$ and $x\leq m$ implies that $x=m$.  

We next prove that the admissible functions of angles $(E,\eta)$ 
from the set $\mathcal{A}_0$ 
are minimal elements in $\mathcal{A}$.   

\begin{lemma}\label{cfoaminimal}
Let $(E,\eta)\in \mathcal{A}_0$.
Let $(E_1,\eta_1)\in\mathcal{A}$ be such that  
$(E_1,\eta_1)\leq (E,\eta)$ in $\mathcal{A}$. 
Then $(E_1,\eta_1)=(E,\eta)$. 
\end{lemma}

\begin{proof}
We consider the case when $(E,\eta)$ is as in 
Theorem \ref{constructionfoainfinite}. 
The remaining case when $(E,\eta)$ is as in 
Theorem \ref{constructionfoafinite} is similar and therefore omitted. 
By construction we have that $\eta(\theta_1)=1$, 
$\eta(\theta_j)<\eta(\theta_k)$ for $1\leq j<k$ 
and $\sin(\eta(\theta_k)\theta_j)=0$ for $1\leq j<k$. 
We need to prove that $\theta_k\in E_1$ and $\eta_1(\theta_k)=\eta(\theta_k)$ for $k=1,2,\dots$.  
We proceed by induction.

Since $(E_1,\eta_1)$ is an admissible function of angles there exists $\theta\in E_1$ such that $\eta_1(\theta)=1$. 
Since $(E_1,\eta_1)\leq (E,\eta)$ we have that $\theta=\theta_j$ for some $j\in\Z^+$ and $\eta(\theta_j)=1$. 
The condition $\eta(\theta_j)=1$ forces $j=1$. 
We have shown that $\theta _1\in E_1$ and  $\eta_1(\theta_1)=\eta(\theta_1)=1$. 

Assume next that $\theta_j\in E_1$ and $\eta_1(\theta_j)=\eta(\theta_j)$ for $1\leq j<k$, where $k\geq2$. 
We shall prove that $\theta_k\in E_1$ and $\eta_1(\theta_k)=\eta(\theta_k)$. 
Let $m=\eta(\theta_k)\in\Z^+$.  
Since $(E_1,\eta_1)$ is an admissible function of angles, 
there exists $\theta\in E_1$ such that $\sin(m\theta)\neq0$ and $m\geq \eta_1(\theta)$. 
Since $(E_1,\eta_1)\leq (E,\eta)$ we have that $\theta=\theta_j$ for some $j\in\Z^+$ with $\eta_1(\theta_j)\geq\eta(\theta_j)$. 
The inequality $\eta(\theta_j)\leq\eta(\theta_k)$ gives that $1\leq j\leq k$.
The condition $\sin(m\theta_j)\neq0$ now forces $j=k$. 
We have shown that $\theta _k\in E_1$ and  $\eta_1(\theta_k)=\eta(\theta_k)$. 
The conclusion of the lemma now follows by induction.
\end{proof}

We are now in position to calculate the minimal elements in $\mathcal{A}$.   

\begin{theorem}\label{minimalelementsA} 
The minimal elements in $\mathcal{A}$ are exactly those that belong to 
the set $\mathcal{A}_0$.
\end{theorem}

\begin{proof}
By Lemma \ref{cfoaminimal} we know that 
the admissible functions of angles $(E,\eta)$ 
from the set $\mathcal{A}_0$ 
are minimal elements in $\mathcal{A}$. 
We proceed to show that every minimal element in $\mathcal{A}$ has such form.

Let $(E_1,\eta_1)\in\mathcal{A}$ be minimal in $\mathcal{A}$. 
By Theorem \ref{lowerboundfoa} there exists $(E,\eta)\in\mathcal{A}_0$ 
such that $(E,\eta)\leq (E_1,\eta_1)$ in $\mathcal{A}$. 
Since $(E_1,\eta_1)$ is minimal in $\mathcal{A}$, 
we conclude that $(E,\eta)=(E_1,\eta_1)$.
\end{proof}

\section*{Acknowledgements}
The research of Jens Wittsten was supported by the Swedish Research Council Grant No.~2019-04878.

\end{document}